\documentclass[11pt]{amsart}
\usepackage{amsmath,amsthm,amsfonts,amssymb,mathrsfs,mathtools}
\usepackage{enumerate}
\usepackage[colorlinks=true, linkcolor=blue, citecolor=magenta, menucolor=black, hypertexnames=false]{hyperref}
\usepackage{graphicx}
\usepackage[up]{caption}
\usepackage{subcaption}
\usepackage{pgf,tikz}
\usepackage[toc,page]{appendix}
\usetikzlibrary{arrows}
\usetikzlibrary{patterns}
\usepackage{color}
\usepackage{comment}
\numberwithin{equation}{section}
\usetikzlibrary{arrows.meta}

\newcommand{\C}{\mathbb{C}}

\newcommand{\N}{\mathbb{N}}
\newcommand{\F}{\mathbb{F}}

\newcommand{\Aut}{\operatorname{Aut}}
\newcommand{\Isom}{\operatorname{Isom}}
\newcommand{\Ch}{\operatorname{Ch}}
\newcommand{\Prob}{\operatorname{Prob}}
\newcommand{\UCP}{\operatorname{UCP}}

\newcommand{\rC}{\operatorname{C}}
\newcommand{\rL}{\operatorname{L}}
\newcommand{\rE}{\operatorname{E}}
\newcommand{\Fix}{\operatorname{Fix}}
\newcommand{\Stab}{\operatorname{Stab}}
\newcommand{\Rad}{\operatorname{Rad}}
\newcommand{\supp}{\operatorname{supp}}

\newcommand{\Pois}{\operatorname{Pois}}
\newcommand{\Tr}{\operatorname{Tr}}
\newcommand{\Char}{\operatorname{Char}}

\newcommand{\cS}{\mathscr{S}}
\newcommand{\cU}{\mathscr{U}}
\newcommand{\dd}{\mathop{}\!\mathrm{d}}
\renewcommand{\bar}[1]{\overline{#1}}

\newtheorem{introresult}{Theorem}
\newtheorem{introcorollary}[introresult]{Corollary}

\newtheorem{theorem}{Theorem}[section]

\newtheorem{proposition}[theorem]{Proposition}
\newtheorem{lemma}[theorem]{Lemma}
\newtheorem{corollary}[theorem]{Corollary}

\theoremstyle{definition}
\newtheorem{definition}[theorem]{Definition}

\newtheorem{remark}[theorem]{Remark}

\begin{document}
	
	\title[Charmenability of twin building lattices]{Charmenability of twin building lattices}
	
	\author{Cyril Houdayer}
	\author{Corentin Le Bars}
	
	\address{\'Ecole normale sup\'erieure \\ D\'epartement de math\'ematiques et applications \\ Universit\'e Paris-Saclay \\ 45 rue d'Ulm \\ 75230 Paris Cedex 05 \\ France}
	\email{cyril.houdayer@ens.psl.eu}
	
	\address{\'Ecole normale sup\'erieure \\ D\'epartement de math\'ematiques et applications \\ 45 rue d'Ulm \\ 75230 Paris Cedex 05 \\ France}
	\email{corentin.le.bars@math.ens.psl.eu}
	
	\thanks{Research supported by ERC Advanced Grant NET 101141693.}
	
	\subjclass[2020]{20F55, 22D25, 37B05, 46L55, 60J50}
	
	\keywords{Kac--Moody groups, twin building lattices, charmenability, charfiniteness, invariant random subgroups, uniformly recurrent subgroups, Howe--Moore property, Poisson boundaries, hyperbolic buildings}
	
	\begin{abstract}
		Let $\Lambda = \mathscr G_A(\F_q)$ be a minimal simply connected split Kac--Moody group of compact-hyperbolic type, and set $\Gamma=\Lambda/Z(\Lambda)$. We prove that $\Gamma$
		is charfinite when $q\geq 11$, or when its Weyl group is $3$-spherical and $q\geq 4$. Here are a few consequences: the only characters of $\Gamma$ are the trivial and regular characters, every invariant random subgroup is supported on $\{e\}$ and $\Gamma$ and these are the only uniformly recurrent subgroups, and every nontrivial unitary $\Gamma$-representation weakly contains the regular representation. 
		
		Our approach relies on a recent result due to Caprace and Thom, showing that the topological completions $G_+$ and $G_-$ associated to $\Gamma$ have the Howe--Moore property \cite{CT26}. We prove a criterion of charmenability for irreducible lattices with trivial amenable radical in products of Howe--Moore groups acting geometrically on proper hyperbolic spaces. The criterion requires the visual boundaries to be Poisson boundaries of suitable random walks and the corresponding lattice actions to be essentially free. 
	\end{abstract}
	
	\maketitle

	\section{Introduction}
	\subsection{Kac--Moody groups: charfiniteness and rigidity}
	
	Let \(\Delta\) be a countable group. Let \(\mathscr P(\Delta)\subseteq\ell^\infty(\Delta)\) be the weak-$\ast$ compact convex set of normalized positive definite functions. The group
	\(\Delta\) acts on \(\mathscr P(\Delta)\) by conjugation. A \emph{trace} on \(\Delta\) is a fixed point of this action. We denote by \(\Tr(\Delta)\) the space of traces, and we call its extreme points \emph{characters}. Following \cite{BBHP22}, the group \(\Delta\) is \emph{charmenable} if every nonempty compact convex \(\Delta\)-invariant subset of \(\mathscr P(\Delta)\) contains a trace and every character is either supported on the amenable radical or generates an amenable von Neumann algebra. It is \emph{charfinite} if, in addition, its amenable radical is finite, it has only finitely many isomorphism classes of unitary representations in each finite dimension, and every amenable character is finite dimensional. We refer to Section~\ref{sec:charmenability} for details. A charmenable group with finite amenable radical and property~{\rm (T)} is charfinite \cite[p.~299]{BBH23}.
	
	Recent works showed that higher-rank lattices in semisimple algebraic groups are charmenable, and those with property~{\rm (T)} are charfinite \cite{BH21,BBHP22,BBH23}. The purpose of this paper is to establish analogous results for lattices in twin buildings. These lattices are non-linear in general, and require new techniques. 
	
	Let \(A=(a_{st})_{s,t\in S}\) be an indecomposable generalized Cartan matrix of compact-hyperbolic type. Its Coxeter system \((W,S)\) is an irreducible cocompact hyperbolic reflection group and every proper standard parabolic subgroup of \(W\) is finite. For a prime power \(q\), let
	\[
	\Lambda=\mathscr G_A(\F_q)
	\]
	be the minimal simply connected split Kac--Moody group, and put \(\Gamma=\Lambda/Z(\Lambda)\). The group $\Lambda$ naturally induces a twin BN-pair, hence acts strongly transitively on a twin building $(X_+, X_-)$, where $X_\varepsilon$ is a building of type $(W,S)$. The group \(\Gamma\) acts faithfully on the two halves \(X_+\) and \(X_-\) of its twin building and embeds as an irreducible lattice in the product \(G_+\times G_-\) of their topological completions \cite{CR09}. Under some conditions on $q$ and the root groups, \(\Gamma\) is infinite, simple, finitely presented and has property~{\rm (T)} \cite{CR09}. A recent result due to Caprace and Thom shows that for $q \geq 11$ or when $W$ is \(3\)-spherical and \(q\geq4\), the groups \(G_+\) and \(G_-\) have the Howe--Moore property \cite{CT26}.
	
	One of the main results of the paper is that $\Gamma$ is charfinite. 
	\begin{introresult}
		\label{thm:KM-main-intro}
		Assume that $q\geq11 $ or that $W$ is \(3\)-spherical and \(q\geq4\). Then \(\Gamma\) is charfinite.
	\end{introresult}
	
	In fact, we prove that $\Gamma$ is charmenable, but for the parameters above Caprace and Thom show that $\Gamma$ has property (T) \cite[Proposition~4.1]{CT26}, hence is charfinite. 
	Theorem~\ref{thm:KM-main-intro} provides a finitely presented infinite simple charfinite group, giving a positive answer to a question raised in \cite[p.~933]{BBHP22}. 
	
	Consider the conjugation action on the Chabauty space \(\operatorname{Sub}(\Gamma)\). An \emph{invariant random subgroup} (\textup{IRS}) is an invariant Borel probability measure, and a \emph{uniformly recurrent subgroup} (\textup{URS}) is a nonempty closed minimal invariant subset. These notions were introduced in \cite{AGV14} and \cite{GW15} respectively.
	
	We draw out a few consequences of Theorem~\ref{thm:KM-main-intro}. 
	\begin{introcorollary}
		\label{cor:KM-main}
		Under the hypotheses of Theorem~\ref{thm:KM-main-intro}, the only characters of \(\Gamma\) are \(1_\Gamma\) and \(\delta_e\). Every $\Gamma$-\textup{IRS} has the form
		\[
		t\,\delta_{\{e\}}+(1-t)\delta_\Gamma,
		\qquad 0\leq t\leq1,
		\]
		and its only \textup{URSs} are \(\{\{e\}\}\) and \(\{\Gamma\}\).
	\end{introcorollary}
	
	In particular $\Gamma$ is character-rigid, answering in the positive a question of Caprace and Thom in \cite[Remark~5.4]{CT26}. The result about \textup{IRSs} recovers \cite[Corollary~E]{CT26} through character rigidity and the fixed-point character of a probability-measure-preserving action, following the arguments of Peterson and Thom \cite{PT16}.
	
	It is also worth noticing that $\Gamma$ has a particularly rigid theory of unitary representations. 
	\begin{introcorollary}
		\label{cor:unitary-intro}
		Under the hypotheses of Theorem~\ref{thm:KM-main-intro}, every nontrivial unitary representation $\rho$ weakly contains the left regular representation
		\(\lambda_\Gamma\).
		
		More precisely, if
		\[
		\mathscr H_0=\mathscr H_\rho\ominus\mathscr H_\rho^\Gamma
		\qquad\text{and}\qquad
		\rho_0=\rho|_{\mathscr H_0},
		\]
		then either \(\mathscr H_0=\{0\}\), or \(\rho_0\) is weakly mixing, nonamenable and weakly contains \(\lambda_\Gamma\). In the latter case, the canonical map
		\[
		\Theta_{\rho_0,\lambda}:
		\rC^*_{\rho_0}(\Gamma)\longrightarrow\rC^*_\lambda(\Gamma)
		\]
		is surjective, \(\tau_\Gamma\circ\Theta_{\rho_0,\lambda}\) is the unique
		tracial state on \(\rC^*_{\rho_0}(\Gamma)\), and
		\(\ker(\Theta_{\rho_0,\lambda})\) is its unique proper maximal ideal.
	\end{introcorollary}
	
	These assertions follow from \cite[Proposition~3.7]{BBHP22}, property~{\rm (T)} and the absence of nontrivial finite-dimensional unitary representations of \(\Gamma\). 
	
	\subsection{Poisson--Howe--Moore singularity and a charmenability criterion}
	
	Let \(G\) be a locally compact group, let \((B,\nu)\) be a nonsingular \(G\)-space and let \(\alpha:G\curvearrowright M\) be an action on a von Neumann algebra. A \emph{boundary structure} over \(B\) is a faithful normal \(G\)-equivariant ucp map
	\[
	\rE:M\longrightarrow\rL^\infty(B,\nu).
	\]
	It is \emph{invariant} when its range is contained in \(\C1\). For \(s\in G\), put \(\rE_s=\rE\circ\alpha_s\). We write \(\rE\perp\rE_s\) when the associated measurable state fields are at norm distance \(2\) almost everywhere.
	
	Let \(\mu\in\Prob(G)\). We say that \(s\in G\) satisfies \emph{\(\mu\)-conjugacy escape} when
	\begin{equation}
		\mu^{*n}\bigl(\{h\in G\mid h^{-1}sh\in K\}\bigr)\longrightarrow0
		\qquad\text{for every compact }K\subset G.
		\tag{CE\(_s\)}
		\label{eq:CE-intro}
	\end{equation}
	
	The main analytic input of this paper is an invariant-or-singular dichotomy over arbitrary Poisson boundaries. 
	
	\begin{introresult}
		\label{thm:PHM-intro}
		Let \(G\) be a second countable locally compact group with the Howe--Moore
		property, let \(\mu\in\Prob(G)\) be generating and let
		\((B,\nu)=\Pois(G,\mu)\). Let \(M\) be a von Neumann algebra with separable
		predual endowed with an ergodic continuous action
		\(\alpha:G\curvearrowright M\). If
		\[
		\rE:M\longrightarrow\rL^\infty(B,\nu)
		\]
		is a faithful normal \(G\)-equivariant ucp map and \(s\in G\) satisfies
		\eqref{eq:CE-intro}, then either \(\rE\) is invariant or
		\(\rE\perp\rE_s\).
	\end{introresult}
	
	This theorem replaces the homogeneous-boundary argument of \cite[Proposition~4.14]{BBHP22}, which was the main analytical proof for charmenability of higher rank lattices. There the Poisson boundary is a homogeneous space \(G/P\). In this case, transitivity of the action reduces an equivariant state field to the orbit of one extremal \(P\)-invariant state, while the relative Mautner property forces that state and its translates to be equal or singular. No such reduction is available for the visual boundaries considered here, as the boundaries \(\partial X_\pm\) are not $G_\pm$-transitive. 
	
	Instead, we use stationarity to turn a failure of singularity into matrix coefficients that remain bounded away from zero along escaping conjugates. The Powers--St\o rmer inequality and the Howe--Moore property then produce a nonzero invariant normal functional. By ergodicity and the Poisson-boundary correspondence, \(\rE\) is then forced to be invariant. The factor-extraction theorem \cite[Theorem~5.8]{BBHP22} and the charmenability criterion \cite[Proposition~4.16]{BBHP22} require no transitivity and remain essential tools here.
	
	Recall that a probability measure on a locally compact group is called \emph{admissible} when its support generates the group as a closed semigroup and some convolution power is non-singular with the Haar measure. We can now state the criterion that we use in order to show that some Kac--Moody lattices are charfinite. 
	
	\begin{introresult}
		\label{thm:twin-main}
		Let \(G_1\) and \(G_2\) be second countable locally compact groups with the Howe--Moore property. Let \(\Gamma<G_1\times G_2\) be a lattice with dense injective coordinate projections \(\pi_i:\Gamma\to G_i\) and trivial amenable radical. Suppose that, for each \(i\in\{1,2\}\), there is an admissible probability measure \(\mu_i\) on \(G_i\) such that \(\pi_i(\gamma)\) satisfies \(\mu_i\)-conjugacy escape for every \(\gamma\in\Gamma\setminus\{e\}\). Then \(\Gamma\) is charmenable. If \(\Gamma\) has property~{\rm (T)}, then it is charfinite. 
	\end{introresult}
	
	The proof of Theorem~\ref{thm:twin-main} is straightforward: Theorem~\ref{thm:PHM-intro} gives the invariant-or-singular dichotomy in each factor. Factor extraction passes it to the lattice, and \cite[Proposition~4.16]{BBHP22} gives charmenability. Property~{\rm (T)} then gives charfiniteness.

	Our goal is then to prove that $\Gamma$ satisfies conjugacy escape on both factors. This property follows from essential freeness of the corresponding harmonic boundary actions. Indeed, suppose that \(G\) acts on a proper geodesic Gromov-hyperbolic space \(X\) and that a \(\mu\)-random walk satisfies \(H_n o\to\xi\in\partial X\) almost surely, where \(\xi\) has law \(\nu\). If \(H_n^{-1}sH_n\) remains in a fixed compact set along a subsequence, then \(H_n o\) and \(sH_n o\) remain at bounded distance along that subsequence. Hence \(s\xi=\xi\). Lemma~\ref{lem:geometric-CE} makes this argument precise and gives
	\begin{equation}
		\nu\bigl(\Fix_{\partial X}(s)\bigr)=0
		\quad\Longrightarrow\quad
		s\text{ satisfies \(\mu\)-conjugacy escape}.
		\label{eq:fixed-points-imply-escape-intro}
	\end{equation}

	Our approach to proving Theorem~\ref{thm:KM-main-intro} is therefore to construct a random walk $(Z_n)$ on $\Gamma$ in such a way that for $\varepsilon \in \{+, -\}$, there exists an admissible measure $\mu_\varepsilon$ on $G_\varepsilon$ for which: 
	\begin{itemize}
		\item the Poisson boundary of $(G_\varepsilon, \mu_\varepsilon) $ is realized as $(\partial X_\varepsilon, \nu_\varepsilon)$, 
		\item $(\pi_\varepsilon(Z_n))$ and the $\mu_\varepsilon $-random walk on $G_\varepsilon$ induce the same Markov chain on the chambers $\Ch(X_\varepsilon)$, so that $\nu_\varepsilon$ is also the hitting measure of $(\pi_\varepsilon(Z_n))$, 
		\item the $\Gamma$-action on $(\partial X_\varepsilon, \nu_\varepsilon)$ is essentially free. 
	\end{itemize}
	
	We give more details on this construction in the next subsection. 
	\subsection{The random walk on the twin-building lattice}
	
	Write the two ambient factors as \(G_+\) and \(G_-\), with projections
	\(\pi_\varepsilon:\Gamma\to G_\varepsilon\) for
	\(\varepsilon\in\{+,-\}\). In our case, these projections are injective and dense, and \(G_\varepsilon\) acts geometrically and nonelementarily on a proper geodesic Gromov-hyperbolic space \(X_\varepsilon\). For \(s\in S\), let
	\[
	L_s=\langle U_{\alpha_s},U_{-\alpha_s}\rangle\leq\Gamma
	\]
	be the finite rank-one subgroup associated to the root $\alpha_s$, and let \(u_{L_s}\) be its uniform probability measure. Choose \(r_s>0\), with \(\sum_{s\in S}r_s=1\), and set
	\[
	\sigma=\sum_{s\in S}r_su_{L_s}.
	\]
	If \(P_s^\varepsilon<G_\varepsilon\) denotes the stabilizer of the fundamental \(s\)-panel and \(m_{P_s^\varepsilon}\) its normalized Haar measure, define
	\[
	\mu_\varepsilon
	=\sum_{s\in S}r_sm_{P_s^\varepsilon}
	\qquad
	(\varepsilon\in\{+,-\}).
	\]
	Lemma~\ref{lem:ef-same-chamber-walk} below shows that the \(\sigma\)-walk on \(\Gamma\) and the \(\mu_\varepsilon\)-walk on \(G_\varepsilon\) induce the same Markov chain on
	\(\operatorname{Ch}(X_\varepsilon)\). Moreover, $(\partial X_\varepsilon,\nu_\varepsilon)$ is the Poisson--Furstenberg boundary of $(G_\varepsilon,\mu_\varepsilon)$. In Section~\ref{sec:essential-freeness}, we use the twinning between $X_\varepsilon$ and $X_{-\varepsilon}$ to prove that the induced $\Gamma$-action is essentially free. 
	
	Recall that a probability measure \(\eta\) on a countable group is \emph{\(\rC^*\)-simple} when the canonical trace is the unique \(\eta\)-stationary state on the reduced group \(\rC^*\)-algebra for the conjugation action. As a side result, we observe that $\sigma$ is a concrete \(\rC^*\)-simple measure on $\Gamma $. 
	
	\begin{introcorollary}
		\label{cor:cstar-measure-intro}
		Under the hypotheses of Theorem~\ref{thm:KM-main-intro}, the measure $\sigma$ defined above is symmetric, finitely supported, generating and \(\rC^*\)-simple. 
	\end{introcorollary}
	
	Indeed, \((\partial X_+,\nu_+)\) is a uniquely \(\sigma\)-stationary \(\sigma\)-boundary, and its action is essentially free, hence the criterion of Hartman and Kalantar applies \cite[Theorem~4.12]{HK23}. In fact, we recover one of the results of Rybak \cite{Ryb26} that \(\Gamma\) is completely selfless, see Corollary~\ref{cor:KM-selfless}. 
	
	\subsection{Outline of the strategy and organization of the paper}
	The proof may be summarized as follows. Harmonic essential freeness and boundary convergence give conjugacy escape in each factor. Theorem~\ref{thm:PHM-intro} gives the factorwise invariant-or-singular dichotomy. Factor extraction passes this dichotomy to the irreducible lattice, and the criterion of \cite{BBHP22} gives charmenability. Property~{\rm (T)} gives charfiniteness.
	
	Section~\ref{sec:lattice-boundaries} recalls the boundary-structure criterion and factor extraction from \cite{BBHP22}. The proof of Theorem~\ref{thm:PHM-intro} occupies Section~\ref{sec:PHM}. The twin-root-datum and lattice framework is developed in Section~\ref{sec:twin}. Section~\ref{sec:essential-freeness} constructs the panel walks, proves essential freeness and derives the results stated above.

	\subsection*{Use of AI tools}
	This project started in spring 2026. We knew that the Poisson boundaries of \(G_+\) and \(G_-\) could be realized as the visual boundaries of the corresponding hyperbolic buildings, but that these spaces were not homogeneous. We therefore needed to replace the relative Mautner argument of \cite{BBHP22}. We did not find a fruitful strategy until the work of Caprace and Thom \cite{CT26} appeared, and our initial idea was to adapt their proof of the Howe--Moore property to obtain the required singularity-or-invariance dichotomy. 
	
	We worked with GPT-5.6 Sol in order to make this adaptation, but the discussion led to a simpler approach using the Howe--Moore property directly and reduced the problem to harmonic essential freeness for well-chosen random walks. The construction of the random walks in this paper appeared natural and is due to the authors. In order to prove essential freeness, part of our input was to use the recurrence of the codistance chain given in Proposition~\ref{prop:ef-codistance-recurrence}. After some interaction, GPT-6 Astra suggested using Lemma~\ref{lem:ef-conjugation-martingale} and provided a proof. The combination of this with Proposition~\ref{prop:ef-codistance-recurrence} gave the desired essential freeness. Aside from the arguments mentioned above, the rest of the ideas and the strategy are entirely due to the authors. GPT-5.6 Sol and GPT-6 Astra were then used to draft parts of the arguments and review some proofs. The authors checked and rewrote most of the arguments and take full responsibility for the mathematical content and the final text.

	\section{Irreducible lattices and boundary structures}\label{sec:lattice-boundaries}
	
	We recall the boundary-structure criterion for charmenability and then reduce
	its invariant-or-singular dichotomy to the ambient factors of an irreducible
	lattice.
	
	\subsection{Charmenability}\label{sec:charmenability}
	
	Let \(\Lambda\) be a countable discrete group. We denote by \(\mathscr P(\Lambda)\) the weak-$\ast$ compact convex space of positive definite functions \(\varphi:\Lambda\to\C\) normalized so that \(\varphi(e)=1\). For \(\varphi\in\mathscr P(\Lambda)\), we denote its GNS triple by \((\pi_\varphi,\mathscr H_\varphi,\xi_\varphi)\), so that
	\(\varphi(g)=\langle\pi_\varphi(g)\xi_\varphi,\xi_\varphi\rangle\) for every
	\(g\in\Lambda\).
	The conjugation action \(\Lambda\curvearrowright\mathscr P(\Lambda)\) is defined by
	\((\gamma\varphi)(g)=\varphi(\gamma^{-1}g\gamma)\) for every
	\(\gamma,g\in\Lambda\).
	A fixed point for this action is called a \emph{trace}. We denote by \(\Tr(\Lambda)\subset\mathscr P(\Lambda)\) the weak-$\ast$ compact convex space of all traces. The GNS correspondence identifies \(\Tr(\Lambda)\) with the tracial state space of the full group \(\rC^*\)-algebra \(\rC^*(\Lambda)\). For \(\tau\in\Tr(\Lambda)\), we put \(M_\tau=\pi_\tau(\Lambda)''\). An extreme point of \(\Tr(\Lambda)\) is called a \emph{character}. We denote by \(\Char(\Lambda)\subset\Tr(\Lambda)\) the set of all characters. Thus \(\tau\in\Tr(\Lambda)\) belongs to \(\Char(\Lambda)\) if and only if \(M_\tau\) is a factor. Any icc group \(\Lambda\) admits the trivial character \(1_\Lambda\) and the regular character \(\delta_e\).
	
	Following \cite{BBHP22}, the group \(\Lambda\) is \emph{charmenable} when the following properties hold:
	\begin{enumerate}[$(\rm i)$]
		\item every nonempty compact convex \(\Lambda\)-invariant subset of
		\(\mathscr P(\Lambda)\) contains an element of \(\Tr(\Lambda)\),
		\item every character \(\tau\in\Char(\Lambda)\) is either supported on the amenable radical \(\Rad(\Lambda)\), or \(M_\tau\) is amenable.
	\end{enumerate}
	A unitary representation is weakly mixing if it has no nonzero
	finite-dimensional subrepresentation.
	Given unitary representations \(\pi_i:\Lambda\to\cU(\mathscr H_i)\), we say that
	\(\pi_2\) is \emph{weakly contained} in \(\pi_1\) if the map
	\(\pi_1(\gamma)\mapsto\pi_2(\gamma)\), for \(\gamma\in\Lambda\), extends to a
	\(*\)-homomorphism
	\(\rC^*_{\pi_1}(\Lambda)\to\rC^*_{\pi_2}(\Lambda)\). Following \cite{BV93}, a
	unitary representation \(\pi:\Lambda\to\cU(\mathscr H_\pi)\) is
	\emph{amenable} if \(1_\Lambda\) is weakly contained in
	\(\pi\otimes\bar\pi\). We call \(\tau\in\Char(\Lambda)\) amenable if its GNS
	representation \(\pi_\tau\) is amenable. It is von Neumann amenable if
	\(M_\tau\) is amenable. Von Neumann amenability implies amenability.
	
	A charmenable group \(\Lambda\) is \emph{charfinite} if its amenable radical
	is finite, it has only finitely many isomorphism classes of unitary
	representations in each finite dimension, and every amenable character
	\(\tau\in\Char(\Lambda)\) is finite dimensional.
	
	\subsection{Boundary structures and stationary states}\label{sec:boundary-structures}
	
	Let \(G\) be a second countable locally compact group, and let \(M\) and
	\(N\) be von Neumann algebras endowed with continuous \(G\)-actions.
	Assume that \(M\) has separable predual, and denote its action by
	\(\alpha:G\curvearrowright M\). The action \(\alpha\) is \emph{ergodic} if
	\(M^G=\C1\). We write \(\UCP_G^\sigma(M,N)\) for the set
	of all normal \(G\)-equivariant unital completely positive maps from \(M\)
	to \(N\).
	
	\begin{samepage}
		\begin{definition}\label{def:boundary-structure}
			A \((G,N)\)-\emph{structure} on \(M\) is a map
			\(\rE\in\UCP_G^\sigma(M,N)\). It is \emph{faithful} if
			\(\rE(x^*x)=0\) implies \(x=0\).
			When \(N=\rL^\infty(B,\nu)\), we call \(\rE\) a
			\emph{boundary structure} over \((B,\nu)\). It is \emph{invariant}
			if its range is contained in \(\C1\).
		\end{definition}
	\end{samepage}
	
	If the action \(G\curvearrowright(B,\nu)\) is ergodic, invariance of \(\rE\) is equivalent to
	\(\rE\circ\alpha_g=\rE\) for every \(g\in G\). In general, the set
	\(\UCP_G^\sigma(M,\rL^\infty(B,\nu))\) may be empty.
	
	Suppose that \(\mu\in\Prob(G)\) is generating and that \((B,\nu)\) is the
	Poisson boundary of \((G,\mu)\). A normal state \(\varphi\in M_*\) is
	\(\mu\)-stationary if
	\(\varphi(x)=\int_G\varphi(\alpha_{g^{-1}}(x))\,\dd\mu(g)\) for every
	\(x\in M\).
	Writing \(\cS_\mu(M)\) for the set of normal \(\mu\)-stationary states on
	\(M\), integration induces an affine bijection
	\begin{equation}\label{eq:PB-correspondence}
		\UCP_G^\sigma\bigl(M,\rL^\infty(B,\nu)\bigr)
		\longrightarrow
		\cS_\mu(M),
		\qquad
		\rE\longmapsto\nu\circ\rE.
	\end{equation}
	See \cite[Proposition~4.2]{BBHP22}.
	
	The \emph{support} of a normal state \(\varphi\) is the least projection
	\(p\in M\) such that \(\varphi(p)=1\). If \(\varphi\) is
	\(\mu\)-stationary, the argument of \cite[Lemma~2.1]{BH21} shows that
	\(p\) is \(G\)-invariant. Thus, when \(G\curvearrowright M\) is ergodic,
	we have \(p=1\), so \(\varphi\) is faithful. In this case, by
	\cite[Proposition~2.7(2)]{BBHP22}, the two sets in
	\eqref{eq:PB-correspondence} are either both empty or both singletons.
	
	Given a boundary structure \(\rE:M\to\rL^\infty(B,\nu)\), choose a separable ultraweakly dense unital \(G\)-invariant
	\(\rC^*\)-subalgebra \(A\subset M\) on which the action is norm continuous,
	using the regularization argument in the proof of \cite[Proposition~XIII.1.2]{Ta03b}.
	The restriction of \(\rE\) is represented by a measurable state field
	\(\theta:B\to\cS(A)\), with \(\rE(x)(b)=\theta_b(x)\) for \(x\in A\).
	For \(g\in G\), put \(\rE_g=\rE\circ\alpha_g\). Two normal ucp maps \(\rE_1,\rE_2:M\to\rL^\infty(B)\) are \emph{singular}, written \(\rE_1\perp\rE_2\), if their state fields satisfy
	\(\|\theta_{1,b}-\theta_{2,b}\|=2\) for \(\nu\)-almost every \(b\in B\).
	This is independent of the compact model \(A\)
	\cite[Proposition~4.10]{BBHP22}. The structure is \(G\)-singular if the maps
	\((\rE_g)_{g\in G}\) are pairwise singular.
	
	\subsection{A criterion for charmenability}
	
	The criterion involves two properties of nonsingular actions. We use the
	conditional-expectation characterization of Zimmer amenability
	\cite[Theorems~A and~3.4]{AEG94} and the definition of metric ergodicity in
	\cite[Definition~3.2]{BF20}.
	
	\begin{definition}\label{def:amenability-metric-ergodicity}
		Let \(G\curvearrowright(X,\nu)\) be a nonsingular action on a standard
		probability space.
		\begin{enumerate}[$(\rm i)$]
			\item The action is \emph{amenable in Zimmer's sense} if there exists
			a \(G\)-equivariant conditional expectation
			\[
			P:\rL^\infty(G\times X,m_G\otimes\nu)
			\longrightarrow\rL^\infty(X,\nu),
			\]
			where \(m_G\) is a left Haar measure and \(G\) acts diagonally by
			\(g\cdot(h,x)=(gh,gx)\). We identify \(\rL^\infty(X,\nu)\) with
			the functions independent of the \(G\)-coordinate. The expectation
			\(P\) is not required to be normal.
			\item The action is \emph{metrically ergodic} if every measurable
			\(G\)-equivariant map from \(X\) to a separable metric space on which
			\(G\) acts continuously by isometries is essentially constant.
		\end{enumerate}
	\end{definition}
	
	Observe that metric ergodicity implies ergodicity by taking the target to be \(\{0,1\}\) with the trivial action.
	
	The following criterion uses these two properties to turn an
	invariant-or-singular dichotomy for boundary structures into charmenability.
	
	\begin{proposition}\label{prop:charm-criterion}
		Let \(\Gamma\) be a countable discrete group with trivial amenable radical,
		and let \((B,\nu)\) be a separable amenable and metrically ergodic
		\(\Gamma\)-space. Suppose that every faithful
		\((\Gamma,\rL^\infty(B))\)-structure on a von Neumann algebra with
		separable predual and an ergodic \(\Gamma\)-action is either invariant or
		\(\Gamma\)-singular. Then \(\Gamma\) is charmenable.
	\end{proposition}
	
	This is \cite[Proposition~4.16]{BBHP22}. The criterion only requires a
	dichotomy for the boundary structures that exist. It does not require the
	existence of a noninvariant structure. We shall also use the fact that
	singularity on a von Neumann subalgebra implies singularity on the ambient
	algebra \cite[Remark~4.12]{BBHP22}.
	
	\subsection{Reduction to the ambient factors}
	
	Let \(G_1\) and \(G_2\) be second countable locally compact groups, put
	\(G=G_1\times G_2\), choose admissible measures
	\(\mu_i\in\Prob(G_i)\), and denote their Poisson boundaries by
	\((B_i,\nu_i)\). Put
	\((B,\nu)=(B_1\times B_2,\nu_1\otimes\nu_2)\).
	
	Each \((B_i,\nu_i)\) admits a standard probability-space model and is
	amenable and metrically ergodic as a \(G_i\)-space by
	\cite[Proposition~2.2(iii) and Theorem~2.7]{BF14}.
	These two properties pass to the product \(G\)-space \((B,\nu)\) by
	\cite[Lemma~3.6]{BF20} and to its restriction to any lattice
	\(\Gamma<G\) by \cite[Lemma~3.5]{BF20}.
	Thus \((B,\nu)\) satisfies the boundary hypotheses of
	Proposition~\ref{prop:charm-criterion}.
	
	Now let \(\Gamma<G\) be a lattice with dense coordinate projections.
	For a \((\Gamma,\rL^\infty(B))\)-structure \((M,\rE)\), the factor-extraction theorem
	\cite[Theorem~5.8]{BBHP22} provides von Neumann subalgebras \(M_i\subset M\) of
	\(G_i\)-continuous elements, endowed with extended \(G_i\)-actions, such that
	\(\rE|_{M_i}\) is a normal \(G_i\)-equivariant ucp map into
	\(\rL^\infty(B_i)\). If \(\rE\) is faithful and noninvariant, then
	some \(i\in\{1,2\}\) satisfies
	\begin{equation}\label{eq:factor-extraction}
		M_i\neq\C1,\qquad
		\rE(M_i)\subset \rL^\infty(B_i),\qquad
		\rE|_{M_i}\text{ is not }G_i\text{-invariant}.
	\end{equation}
	No transitivity assumption on \(B_i\) occurs in this theorem.
	
	The next result reduces the lattice-level invariant-or-singular dichotomy to
	the corresponding dichotomy for each factor Poisson boundary.
	
	\begin{theorem}\label{thm:product-reduction}
		Let \(G_1\) and \(G_2\) be second countable locally compact groups. Let
		\(\Gamma<G_1\times G_2\) be a lattice with dense and injective coordinate
		projections \(\pi_i:\Gamma\to G_i\). For each \(i\in\{1,2\}\), let
		\((B_i,\nu_i)\) be the Poisson boundary of an admissible measure \(\mu_i\)
		on \(G_i\). Assume that
		\begin{enumerate}[$(\rm i)$]
			\item \(\Rad(\Gamma)=\{e\}\),
			\item for every \(i\in\{1,2\}\), every faithful
			\((G_i,\rL^\infty(B_i))\)-structure \(\rE\) on a von Neumann algebra with
			separable predual and an ergodic \(G_i\)-action is either invariant or satisfies
			\[
			\rE\perp \rE_{\pi_i(\gamma)}
			\quad\text{for every }\gamma\in\Gamma\setminus\{e\}.
			\]
		\end{enumerate}
		Then \(\Gamma\) is charmenable. If moreover \(\Gamma\) has
		property~{\rm (T)}, then it is charfinite.
	\end{theorem}
	
	\begin{proof}
		Put \((B,\nu)=(B_1\times B_2,\nu_1\otimes\nu_2)\).
		As explained above, this is a separable amenable and metrically ergodic
		\(\Gamma\)-space. By Proposition~\ref{prop:charm-criterion}, it suffices
		to prove the invariant-or-singular dichotomy for boundary structures over
		\((B,\nu)\).
		
		Let \(M\) be a von Neumann algebra with separable predual endowed with
		an ergodic \(\Gamma\)-action, and let \(\rE\) be a faithful
		\((\Gamma,\rL^\infty(B))\)-structure on \(M\). Suppose that \(\rE\) is not
		invariant. By \eqref{eq:factor-extraction}, there is \(i\in\{1,2\}\) such that
		\(\rE_i=\rE|_{M_i}\) is a noninvariant faithful
		\((G_i,\rL^\infty(B_i))\)-structure. The action \(G_i\curvearrowright M_i\) is
		ergodic. Indeed, the extended \(G_i\)-action restricts along \(\pi_i\) to
		the original \(\Gamma\)-action, so
		every \(x\in M_i\) fixed by \(G_i\) is fixed by \(\Gamma\) and therefore
		scalar.
		
		Let \(\gamma,\eta\in\Gamma\) be distinct. Assumption~$(\rm ii)$ and injectivity of
		\(\pi_i\) give
		\(\rE_i\perp(\rE_i)_{\pi_i(\eta\gamma^{-1})}\).
		Common precomposition by \(\alpha_{\pi_i(\gamma)}\) shows that
		\((\rE_i)_{\pi_i(\gamma)}\perp(\rE_i)_{\pi_i(\eta)}\).
		
		Let \(p_i:B\to B_i\) be the coordinate projection and let \(p_i^*\)
		denote pullback of functions. The compatibility of the actions gives
		\[
		\rE_\kappa|_{M_i}
		=p_i^*\circ(\rE_i)_{\pi_i(\kappa)}
		\qquad(\kappa\in\{\gamma,\eta\}).
		\]
		Since \((p_i)_*\nu=\nu_i\), the state-field characterization in
		\cite[Proposition~4.10]{BBHP22} gives
		\(\rE_\gamma|_{M_i}\perp\rE_\eta|_{M_i}\).
		It follows from \cite[Remark~4.12]{BBHP22} that
		\(\rE_\gamma\perp\rE_\eta\). Since \(\gamma\) and \(\eta\) were
		arbitrary distinct elements, \(\rE\) is \(\Gamma\)-singular.
		Proposition~\ref{prop:charm-criterion} now gives charmenability.
		The assertion about charfiniteness follows from
		\cite[p.~299]{BBH23}.
	\end{proof}
	
	\section{Poisson--Howe--Moore singularity}\label{sec:PHM}
	
	Throughout this section, \(G\) is a second countable locally compact group,
	\(\mu\in\Prob(G)\) is generating and \((B,\nu)\) is the Poisson boundary of
	\((G,\mu)\). Let \(\alpha:G\curvearrowright M\) be a continuous action.
	We use the conventions
	\[
	(g\varphi)(x)=\varphi(\alpha_{g^{-1}}(x)),
	\qquad
	(g f)(b)=f(g^{-1}b).
	\]
	Let \(\rE:M\to\rL^\infty(B)\) be a normal \(G\)-equivariant ucp map, let
	\(A\subset M\) be a separable unital \(G\)-invariant \(\rC^*\)-subalgebra on
	which the action is norm continuous, and let
	\(\theta:B\to\cS(A)\) represent \(\rE|_A\). Equivariance gives, for every
	fixed \(g\in G\),
	\begin{equation}\label{eq:field-equivariance}
		\theta_{gb}=g\theta_b
		\quad\text{for \(\nu\)-almost every \(b\in B\)}.
	\end{equation}
	
	\subsection{Conjugacy escape and singularity}
	
	\begin{definition}\label{def:CE}
		An element \(s\in G\) satisfies \emph{\(\mu\)-conjugacy escape} if
		\[
		\mu^{*n}\bigl(\{h\in G\mid h^{-1}sh\in K\}\bigr)\longrightarrow0
		\]
		for every compact subset \(K\subset G\).
	\end{definition}
	
	The singularity criterion of \cite[Proposition~4.14]{BBHP22} uses a
	homogeneous boundary \(G/P\). An extremal boundary structure is then
	determined by one extremal \(P\)-invariant state on a separable
	\(G\)-\(\rC^*\)-algebra. When \(P\) is stably self-normalizing and has the
	relative Mautner property, that state and each of its translates are either
	equal or singular. If the structure is noninvariant, the stabilizer \(Q\) of
	the state is a proper closed subgroup containing \(P\). A null fixed-point
	set for \(s\) on \(G/Q\) then gives \(\rE\perp\rE_s\).
	
	Theorem~\ref{thm:PHM} replaces this use of the relative Mautner phenomenon
	with the Howe--Moore property and \(\mu\)-conjugacy escape. The proof works
	with the measurable state field over \(B\) and its average
	\(\varphi=\nu\circ\rE\). If singularity fails, stationarity produces
	conjugates \(k_j\to\infty\) of \(s\) for which
	\(\|\varphi-\varphi\circ\alpha_{k_j}\|\) stays uniformly below \(2\).
	The Powers--St\o rmer inequality and the Howe--Moore property then produce
	a nonzero invariant vector in the standard representation of \(M\).
	Uniqueness of the normal stationary state forces \(\rE\) to be invariant.
	This argument requires neither a homogeneous realization of \(B\) nor
	extremality of the individual states in the boundary field.
	
	\begin{theorem}\label{thm:PHM}
		Assume that \(G\) has the Howe--Moore property. Let \(M\) be a von Neumann
		algebra with separable predual endowed with an ergodic continuous
		\(G\)-action, and let \(\rE:M\to\rL^\infty(B,\nu)\)
		be a faithful normal \(G\)-equivariant ucp map. Fix \(s\in G\) satisfying
		\(\mu\)-conjugacy escape. Then either \(\rE\) is \(G\)-invariant or
		\(\rE\perp\rE_s\).
	\end{theorem}
	
	\begin{proof}
		We prove that failure of singularity forces \(\rE\) to be invariant.
		Choose a separable ultraweakly dense unital
		\(G\)-invariant \(\rC^*\)-subalgebra \(A\subset M\) on
		which the action is norm continuous. Let \(\theta:B\to\cS(A)\) be
		the state field associated with \(\rE|_A\).
		For \(t\in G\), put \(d_t(b)=\|\theta_b-\theta_{t^{-1}b}\|\).
		Suppose that \(\rE\not\perp\rE_s\). Then \(d_s<2\) on a set of positive
		measure. Thus, after discarding a null set, there are
		\(\delta>0\) and a measurable set \(D\subset B\), with \(p=\nu(D)>0\), such
		that \(d_s(b)\leq2-\delta\) for every \(b\in D\).
		
		We first choose escaping conjugates along which this bound persists.
		Since \(G\) is second countable, choose compact sets \(K_j\) such that
		\(K_j\subset\operatorname{int}(K_{j+1})\) and
		\(G=\bigcup_{j\geq1}K_j\). Conjugacy escape allows us to choose strictly
		increasing integers \((n_j)_{j\in\N_{\geq1}}\) such that
		\begin{equation}\label{eq:CE-choice}
			\mu^{*n_j}\bigl(\{h\in G\mid h^{-1}sh\in K_j\}\bigr)<p/2.
		\end{equation}
		Stationarity gives
		\begin{equation}\label{eq:stationarity-D}
			\int_G\nu(h^{-1}D)\,\dd\mu^{*n_j}(h)=p.
		\end{equation}
		Put \(S_j=\{h\in G\mid h^{-1}sh\notin K_j\}\).
		Equations~\eqref{eq:CE-choice} and \eqref{eq:stationarity-D} imply
		\(\int_{S_j}\nu(h^{-1}D)\,\dd\mu^{*n_j}(h)>p/2\). We may therefore choose
		\(h_j\in S_j\) such that \(\nu(h_j^{-1}D)\geq p/2\).
		Put \(k_j=h_j^{-1}sh_j\). Then \(k_j\notin K_j\), so \(k_j\to\infty\).
		
		For every \(j\in\N_{\geq1}\), using \eqref{eq:field-equivariance} and the
		isometric action on \(\cS(A)\), we obtain, for almost every \(b\in B\),
		\begin{align}
			d_{k_j}(b)
			&=\|\theta_b-\theta_{h_j^{-1}s^{-1}h_jb}\|\notag\\
			&=\|\theta_{h_jb}-\theta_{s^{-1}h_jb}\|
			=d_s(h_jb).
			\label{eq:transport-distance}
		\end{align}
		Hence \(d_{k_j}\leq2-\delta\) on \(h_j^{-1}D\).
		
		We now pass from the state field to the normal state
		\(\varphi=\nu\circ\rE\). It is faithful because \(\rE\) is faithful, and
		it is \(\mu\)-stationary by \eqref{eq:PB-correspondence}.
		For each \(j\in\N_{\geq1}\), put
		\(\omega_j=\varphi-\varphi\circ\alpha_{k_j}\in M_*\).
		Normality of \(\omega_j\) and Kaplansky's density theorem give
		\begin{equation}\label{eq:kaplansky-norm}
			\|\omega_j\|
			=\sup\bigl\{|\omega_j(a)|\mid a\in A,\ \|a\|\leq1\bigr\}.
		\end{equation}
		For every \(a\in A\) with \(\|a\|\leq1\), equivariance of the state field
		gives
		\begin{align*}
			\omega_j(a)
			&=\int_B\bigl(\theta_b(a)-\theta_b(\alpha_{k_j}(a))\bigr)\,\dd\nu(b)\\
			&=\int_B\bigl(\theta_b(a)-\theta_{k_j^{-1}b}(a)\bigr)\,\dd\nu(b),\\
			|\omega_j(a)|
			&\leq\int_B\|\theta_b-\theta_{k_j^{-1}b}\|\,\dd\nu(b).
		\end{align*}
		Taking the supremum in \eqref{eq:kaplansky-norm} and then using
		\(\nu(h_j^{-1}D)\geq p/2\), we obtain
		\begin{align}
			\|\varphi-\varphi\circ\alpha_{k_j}\|
			&\leq\int_B\|\theta_b-\theta_{k_j^{-1}b}\|\,\dd\nu(b)\notag\\
			&\leq(2-\delta)\nu(h_j^{-1}D)+2\bigl(1-\nu(h_j^{-1}D)\bigr)\notag\\
			&\leq2-\frac{\delta p}{2}.
			\label{eq:normal-overlap}
		\end{align}
		
		To apply the Howe--Moore property, let \((M,\rL^2(M),J,\rL^2(M)_+)\)
		be the standard form of \(M\), and let
		\(U:G\to\cU(\rL^2(M))\) be the canonical strongly continuous implementation
		of \(\alpha\), as in \cite[Theorems~IX.1.14 and~IX.1.15]{Ta03a}.
		Denote by \(\xi_\varphi\in\rL^2(M)_+\) the unique unit vector implementing
		\(\varphi\) \cite[Theorem~IX.1.2$(\rm iv)$]{Ta03a}. The vector implementing
		\(\varphi\circ\alpha_{k_j}\) is \(U_{k_j}^*\xi_\varphi\). By the
		Powers--St\o rmer inequality \cite[Theorem~IX.1.2$(\rm iv)$]{Ta03a},
		\[
		\|\xi_\varphi-U_{k_j}^*\xi_\varphi\|^2
		\leq\|\varphi-\varphi\circ\alpha_{k_j}\|
		\leq2-\frac{\delta p}{2}.
		\]
		It follows that
		\begin{equation}\label{eq:coefficient-lower}
			\operatorname{Re}\langle\xi_\varphi,U_{k_j}^*\xi_\varphi\rangle
			\geq\frac{\delta p}{4}.
		\end{equation}
		
		Let \(P\) be the orthogonal projection onto \(\rL^2(M)^G\).
		The restriction of \(U\) to \((1-P)\rL^2(M)\) has no invariant vectors.
		Since \(k_j^{-1}\to\infty\), the Howe--Moore property implies
		\[
		\langle\xi_\varphi,U_{k_j}^*\xi_\varphi\rangle
		\longrightarrow\|P\xi_\varphi\|^2.
		\]
		Together with \eqref{eq:coefficient-lower}, this gives
		\(\|P\xi_\varphi\|^2\geq\delta p/4>0\). Hence the vector functional
		\(\psi(x)=\langle xP\xi_\varphi,P\xi_\varphi\rangle\) is a nonzero normal
		positive \(G\)-invariant functional. After normalization,
		it is a normal \(\mu\)-stationary state. By ergodicity and uniqueness of
		the normal stationary state \cite[Proposition~2.7(2)]{BBHP22}, this
		normalized functional equals \(\varphi\). Thus \(\varphi\) is \(G\)-invariant.
		
		The constant map \(x\mapsto\varphi(x)1\) is therefore a boundary structure
		with associated state \(\varphi\). The bijection
		\eqref{eq:PB-correspondence} identifies it with \(\rE\), proving that
		\(\rE\) is \(G\)-invariant.
	\end{proof}
	
	\begin{proof}[Proof of Theorem~\ref{thm:twin-main}]
		For each \(i\in\{1,2\}\), let \((B_i,\nu_i)\) be the Poisson boundary
		of \((G_i,\mu_i)\). The conjugacy-escape assumption and
		Theorem~\ref{thm:PHM} show that every faithful
		\((G_i,\rL^\infty(B_i))\)-structure \(\rE\) on a von Neumann algebra with
		separable predual and an ergodic \(G_i\)-action is either invariant or satisfies
		\[
		\rE\perp\rE_{\pi_i(\gamma)}
		\qquad\text{for every }\gamma\in\Gamma\setminus\{e\}.
		\]
		Theorem~\ref{thm:product-reduction} therefore gives charmenability of
		\(\Gamma\), and charfiniteness when \(\Gamma\) has property~{\rm (T)}.
	\end{proof}
	
	\subsection{A geometric criterion for escape}
	
	The next lemma shows that \(s\) satisfies conjugacy escape whenever the
	boundary limit of the random walk is almost surely not fixed by \(s\).
	
	\begin{lemma}\label{lem:geometric-CE}
		Let \(G\) act continuously by isometries on a proper geodesic
		Gromov-hyperbolic space \(X\). Fix \(o\in X\) and \(s\in G\). Let
		\((\Omega,\mathbf P)\) carry independent \(G\)-valued random variables
		\((Y_n)_{n\in\N_{\geq1}}\) with law \(\mu\), and put
		\(H_n=Y_1\cdots Y_n\) for \(n\in\N_{\geq1}\). Assume that there is a
		measurable random variable
		\(\xi:\Omega\to\partial X\) such that
		\(H_n(\omega)o\to\xi(\omega)\) for almost every \(\omega\in\Omega\), and
		that \(\xi_*\mathbf P=\nu\). If \(\nu(\Fix_{\partial X}(s))=0\), then \(s\)
		satisfies \(\mu\)-conjugacy escape.
	\end{lemma}
	
	\begin{proof}
		Fix a compact set \(K\subset G\). Continuity of the orbit map gives
		\(R_K=\sup_{k\in K}d(o,ko)<\infty\). If \(h\in G\) satisfies
		\(h^{-1}sh\in K\), then
		\(d(ho,sho)=d(o,h^{-1}sho)\leq R_K\). For \(n\in\N_{\geq1}\), set
		\(A_n=\{\omega\in\Omega\mid H_n(\omega)^{-1}sH_n(\omega)\in K\}\).
		
		For almost every \(\omega\), we have \(H_n(\omega)o\to\xi(\omega)\)
		and \(s\xi(\omega)\neq\xi(\omega)\). Fix such a path. If
		\(\omega\in A_n\) for infinitely many \(n\), choose a subsequence
		\((n_j)\) along which
		\[
		d(H_{n_j}(\omega)o,sH_{n_j}(\omega)o)\leq R_K
		\quad\text{for every }j\in\N_{\geq1}.
		\]
		The two sequences converge to \(\xi(\omega)\) and \(s\xi(\omega)\),
		respectively. Since they remain at bounded distance in the hyperbolic
		space \(X\), their boundary limits coincide. This contradicts
		\(s\xi(\omega)\neq\xi(\omega)\).
		
		Thus \(\mathbf 1_{A_n}\to0\) almost surely. Dominated convergence and
		the fact that \(H_n\) has distribution \(\mu^{*n}\) give
		\[
		\mu^{*n}\bigl(\{h\in G\mid h^{-1}sh\in K\}\bigr)
		=\mathbf P(H_n^{-1}sH_n\in K)\longrightarrow0.
		\]
	\end{proof}
	
	\begin{corollary}\label{cor:visual-singularity}
		In the setting of Lemma~\ref{lem:geometric-CE}, assume that \(G\) has the
		Howe--Moore property and that the almost-sure limit map realizes the Poisson
		boundary as \((\partial X,\nu)\). Let \(\Lambda<G\) be a countable subgroup such that
		\[
		\nu(\Fix_{\partial X}(s))=0
		\quad\text{for every }s\in\Lambda\setminus\{e\}.
		\]
		Let \(\rE\) be a faithful \((G,\rL^\infty(\partial X))\)-structure on \(M\),
		and assume that the action \(G\curvearrowright M\) is ergodic.
		If \(\rE\) is not invariant, then \((\rE_s)_{s\in\Lambda}\) is pairwise singular.
	\end{corollary}
	
	\begin{proof}
		Theorem~\ref{thm:PHM} and Lemma~\ref{lem:geometric-CE} give
		\(\rE\perp\rE_s\) for every \(s\in\Lambda\setminus\{e\}\). If
		\(r,t\in\Lambda\) and \(r\neq t\), then \(\rE\perp\rE_{tr^{-1}}\), and
		common precomposition by \(\alpha_r\) gives
		\(\rE_r\perp\rE_{tr^{-1}}\circ\alpha_r=\rE_t\).
	\end{proof}

	\section{Twin-building lattices}\label{sec:twin}
	In this section, we present twin-buildings and Kac--Moody groups, following \cite{CR09} and \cite{CR-RGD}. 
	
	\subsection{Groups with a twin root datum}\label{sec:twin-root-datum}

	Let $(W,S)$ be a Coxeter system of finite rank, and let	$\ell\colon W\to\N$ be its word-length function.  We identify the chambers of the Coxeter complex $\Sigma=\Sigma(W,S)$ with $W$, and denote by $c_0$ the chamber corresponding to $e\in W$. 
	A \emph{root} is a combinatorial half-apartment of $\Sigma$.  We write $\Phi$ for the set of roots and $-\alpha$ for the root complementary to $\alpha$.  The sets of positive and
	negative roots are $\Phi_+=\{\alpha\in\Phi\mid c_0\in\alpha\}$ and $ \Phi_-=-\Phi_+$. For $s\in S$, the wall of the $s$-panel of $c_0$ bounds a unique positive root denoted $\alpha_s$.  We write $\Pi=\{\alpha_s\mid s\in S\}$ for the set of simple roots.
	
	A pair $\{\alpha,\beta\}\subseteq\Phi$ is \emph{prenilpotent} if both
	$\alpha\cap\beta$ and $(-\alpha)\cap(-\beta)$ contain a chamber.  For such a
	pair, define
	\begin{align*}
		[\alpha,\beta]=\bigl\{\gamma\in\Phi\mid
		\alpha\cap\beta\subseteq\gamma \text{ and }(-\alpha)\cap(-\beta)\subseteq-\gamma\bigr\}
	\end{align*}
	and $]\alpha,\beta[=[\alpha,\beta]\setminus\{\alpha,\beta\}$. The interval $[\alpha,\beta]$ is finite.

	\begin{definition}
		\label{def:twin-root-datum}
		Consider a group $\Lambda$ together with a family of subgroups $(U_\alpha)_{\alpha\in\Phi}$.  Set $T=\bigcap_{\alpha\in\Phi}N_\Lambda(U_\alpha)$, $U_+=\langle U_\alpha\mid\alpha\in\Phi_+\rangle$ and $ U_-=\langle U_\alpha\mid\alpha\in\Phi_-\rangle$. 
		The collection $(\Lambda, (U_\alpha)_{\alpha \in \Phi}) $ forms a \emph{twin root datum of type $(W,S)$} if the following axioms are satisfied. 
		
		\begin{enumerate}[$(\rm i)$]
			\item $U_\alpha\neq\{e\}$ for every $\alpha\in\Phi$.
			
			\item For every prenilpotent pair $\{\alpha,\beta\}$, the commutator $[U_\alpha,U_\beta]$ is contained in the group $U_{(\alpha, \beta)}= \langle U_\gamma\mid\gamma\in]\alpha,\beta[\rangle$.
			
			\item For every $s\in S$ and every
			$u\in U_{\alpha_s}\setminus\{e\}$, there exist
			$u',u''\in U_{-\alpha_s}$ such that  $m(u)=u'uu''$ satisfies $ m(u)U_\beta m(u)^{-1}=U_{s(\beta)}$ for all $\beta \in \Phi$.
			
			\item For every $s\in S$, the group $U_{-\alpha_s}$ is not contained in $ U_+$ and $ U_{\alpha_s}$ is not contained in $ U_-$. 
			
			\item $\Lambda= T\Lambda^\dagger$, where $ \Lambda^\dagger= \langle U_\alpha\mid\alpha\in\Phi\rangle$. 
		\end{enumerate}
		The subgroups $U_\alpha$ are the \emph{root groups}.
	\end{definition}
	Define $N=T\langle m(u)\mid u\in U_{\alpha_s}\setminus\{e\},\ s\in S\rangle$, $B_+=TU_+$ and $B_-=TU_-$. Then $T$ is normal in $N$, and the map $ m(u)T\longmapsto s$ extends to an isomorphism $N/T\cong W$. Moreover, $(B_+,N)$ and $(B_-,N)$ are $BN$-pairs with Weyl group $(W,S)$, and together they form a saturated twin $BN$-pair \cite[\S0.4]{CR09}.  For each $w\in W$, choose a representative $\dot w\in N$. The $BN$-pairs give the Bruhat decompositions
	\begin{equation}
		\Lambda=\bigsqcup_{w\in W}B_+\dot wB_+
		=\bigsqcup_{w\in W}B_-\dot wB_-,
		\label{eq:rgd-Bruhat}
	\end{equation}
	and their twinning gives the Birkhoff decomposition
	\begin{equation}
		\Lambda=\bigsqcup_{w\in W}B_+\dot wB_-
		=\bigsqcup_{w\in W}B_-\dot wB_+.
		\label{eq:rgd-Birkhoff}
	\end{equation}
	See \cite[\S5.4]{CR-RGD}.  
	
	\subsection{Twin-buildings}\label{sec:twin-building}
	Let $(W,S)$ be a Coxeter system. We recall the definition of a building.
	\begin{definition}
		A building of type $(W,S)$ is a set $\mathcal{C}$ endowed with a \emph{$W$-valued distance}, that is, a map $\delta\colon\mathcal C\times\mathcal C\longrightarrow W$ satisfying, for all $c,d\in\mathcal C$, all $s\in S$, and all $w\in W$:
		\begin{enumerate}[$(\rm i)$]
			\item
			$\delta(c,d)=e$ if and only if $c=d$.
			
			\item
			If $\delta(c,d)=w$ and $\delta(d,d')=s$, then
			$\delta(c,d')\in\{w,ws\}$. Moreover,
			$\delta(c,d')=ws$ whenever $\ell(ws)=\ell(w)+1$.
			
			\item
			If $\delta(c,d)=w$, then there exists $d'\in\mathcal C$ such that
			$\delta(d,d')=s$ and $\delta(c,d')=ws$.
		\end{enumerate}
	\end{definition} 
	
	Its elements are called \emph{chambers}, and distinct chambers $c,d$ are $s$-adjacent precisely when $\delta(c,d)=s$.  The gallery distance is defined as
	\begin{equation}
		d_{\mathrm{gal}}(c,d)=\ell(\delta(c,d)),
		\label{eq:rgd-gallery-Weyl}
	\end{equation}
	and the type sequences of minimal galleries from $c$ to $d$ are exactly the
	reduced words representing $\delta(c,d)$.
	
	\begin{definition}
		\label{def:twin-building}
		Let $(\mathcal C_+,\delta_+)$ and $(\mathcal C_-,\delta_-)$ be buildings of type $(W,S)$.  A \emph{Weyl codistance}, or \emph{twinning}, is a
		map
		\[
		\delta^*\colon (\mathcal C_+\times\mathcal C_-) \sqcup (\mathcal C_-\times\mathcal C_+)
		\longrightarrow W
		\]
		satisfying the following axioms, and the analogous axioms with the two signs interchanged.  If $x\in\mathcal C_+$, $s\in S$ and $y,z\in\mathcal C_-$, then:
		\begin{enumerate}[$(\rm i)$]
			\item
			$\delta^*(x,y)=\delta^*(y,x)^{-1}$.
			
			\item
			If $\delta^*(x,y)=w$ and $\delta_-(y,z)=s$, then
			$\delta^*(x,z)\in\{w,ws\}$. Moreover,
			$\delta^*(x,z)=ws$ whenever $\ell(ws)=\ell(w)-1$.
			
			\item
			If $\delta^*(x,y)=w$, then there exists $z\in\mathcal C_-$ such
			that $\delta_-(y,z)=s$ and $\delta^*(x,z)=ws$.
		\end{enumerate}
		The triple $(\mathcal C_+,\mathcal C_-,\delta^*)$ is called a \emph{twin building}.  Chambers $x\in\mathcal C_+$ and $y\in\mathcal C_-$ are \emph{opposite} if $\delta^*(x,y)=e$. 
	\end{definition}

	Let $(\Lambda, (U_\alpha)_{\alpha \in \Phi}) $ be a twin root datum of type $(W,S)$, as in Section~\ref{sec:twin-root-datum}. To the BN-pairs $(B_+, N)$ and $(B_-,N)$, one can associate the buildings $X_+$ and $X_-$ respectively. For $\varepsilon\in\{+,-\}$, the set of chambers $\Ch(X_\varepsilon) $ of $X_\varepsilon$ is identified with $ \Lambda/B_\varepsilon$, and we denote its fundamental chamber by $c_\varepsilon = B_\varepsilon$.  The Weyl distance $\delta_\varepsilon$ on the set of chambers $\Ch(X_\varepsilon) $ is characterized by
	\begin{equation}
		\delta_\varepsilon(gB_\varepsilon,hB_\varepsilon)=w
		\quad\Longleftrightarrow\quad
		g^{-1}h\in B_\varepsilon\dot wB_\varepsilon.
		\label{eq:rgd-Weyl-cosets}
	\end{equation}
	The disjointness in \eqref{eq:rgd-Bruhat} shows that this is well-defined. 
	The standard apartments are $\mathcal A_\varepsilon=Nc_\varepsilon
	=\{nB_\varepsilon\mid n\in N\}$ and their simultaneous $\Lambda$-translates form the system of twin apartments.  Each $\mathcal A_\varepsilon$ is canonically a Coxeter complex of
	type $(W,S)$, and there is a natural twinning between $X_+$ and $X_-$ given by
	\begin{equation}
		\delta^*(gB_+,hB_-)=w
		\quad\Longleftrightarrow\quad
		g^{-1}h\in B_+\dot wB_-. 
		\label{eq:rgd-codistance-cosets}
	\end{equation}
	Throughout the paper, we keep this convention for the multiplication formula. Note that it is the inverse of the one used in \cite[\S5.6]{CR-RGD}, but this choice is harmless. The codistance $\delta^*$ is $\Lambda$-invariant. In particular, $B_\varepsilon$ is transitive on the chambers opposite $c_\varepsilon$.

	\subsection{Panels and rank-one subgroups}
	
	For $s\in S$ and $\varepsilon \in \{+,-\}$, define the standard minimal parabolic subgroups
	\begin{equation*}
		\mathsf P_s^\varepsilon
		=B_\varepsilon\sqcup B_\varepsilon\dot sB_\varepsilon.
	\end{equation*}
	The fundamental $s$-panels are $\mathcal P_s^\varepsilon=\mathsf P_s^\varepsilon/B_\varepsilon$, and their translates are all the $s$-panels of $X_\varepsilon$. The corresponding rank-one subgroup is $L_s=\langle U_{\alpha_s},U_{-\alpha_s}\rangle$. The group $L_s$ acts $2$-transitively on each fundamental $s$-panel.  On $\mathcal P_s^+$, the group $U_{\alpha_s}$ fixes $c_+$ and acts sharply transitively on $\mathcal P_s^+\setminus\{c_+\}$. On $\mathcal P_s^-$ the analogous statement holds with $U_{-\alpha_s}$ and $c_-$. 
	
	Assume that the root groups $U_\alpha$ are finite. By
	Definition~\ref{def:twin-root-datum}$(\rm iii)$, $U_{\alpha_s}$ and $U_{-\alpha_s}$ are conjugate, and we set 
	\begin{equation*}
		q_s=|U_{\alpha_s}|=|U_{-\alpha_s}|.
	\end{equation*}
	By sharp transitivity, $|\mathcal P_s^+|=|\mathcal P_s^-|=q_s+1$, so that $X_+$ and $X_-$ are locally finite and type-regular.

	\begin{lemma}
		\label{lem:rgd-rank-one-finite}
		If the root groups are finite, then every $L_s$ is finite.  Moreover,
		\[
		\Lambda^\dagger=\langle L_s\mid s\in S\rangle.
		\]
	\end{lemma}
	
	\begin{proof}
		The kernel of the action of $L_s$ on its rank-one building is $Z(L_s)$,
		as shown in the proof of \cite[Lemma~10]{CR09}. Since the panel has $q_s+1$
		points, $L_s/Z(L_s)$ is finite. Schur's theorem therefore implies that
		$[L_s,L_s]$ is finite.  The abelianization of $L_s$ is generated by the
		images of the two finite groups $U_{\alpha_s}$ and $U_{-\alpha_s}$, hence
		is finite as well. It follows that $L_s$ is finite.
		
		By Definition~\ref{def:twin-root-datum}$(\rm iii)$, the elements $m(u)\in L_s$ lift the simple
		reflections and conjugate simple root groups to all their $W$-translates.
		Since every root is a $W$-translate of a simple root, the groups $L_s$
		generate every $U_\alpha$, and hence generate $\Lambda^\dagger$.
	\end{proof}
	
	For a reduced expression $w=s_1\cdots s_n$, put $q_w=q_{s_1}\cdots q_{s_n}$. This number is independent of the chosen reduced expression, and is the cardinality of the $w$-sphere $\{d\in\Ch(X_\varepsilon)\mid\delta_\varepsilon(c_\varepsilon,d)=w\}$ about $c_\varepsilon$. 
	
	The following codistance rule will be useful later when we relate the twin-building geometry and the Markov chain associated to a random walk. 
	
	\begin{lemma}
		\label{lem:rgd-panel-codistance}
		Fix $c_-\in\Ch(X_-)$, let $\mathcal P\subseteq\Ch(X_+)$ be an $s$-panel, and let $d\in\mathcal P$.  Put $w=\delta^*(d,c_-)$.  Then
		\begin{enumerate}[$(\rm i)$]
			\item If $\ell(sw)=\ell(w)+1$, exactly one chamber of $\mathcal P$ has
			codistance $sw$ from $c_-$, and the other $q_s$ chambers have
			codistance $w$.
			\item If $\ell(sw)=\ell(w)-1$, the chamber $d$ is the unique chamber of
			$\mathcal P$ having codistance $w$ from $c_-$, and every chamber of
			$\mathcal P\setminus\{d\}$ has codistance $sw$.
		\end{enumerate}
		A similar statement holds for an $s$-panel in the building $X_-$, but with $ws$ in place of $sw$.
	\end{lemma}
	
	\begin{proof}
		First, recall that $|\mathcal P|=q_s+1$. By \cite[Lemma~5.149]{AB08}, there exists a unique chamber in $\mathcal P$ realizing the maximal value of $\{\ell(\delta^*(d',c_-)) \mid d' \in \mathcal P\}$. If $\ell(sw)=\ell(w)+1$, this maximal value is attained for the codistance $sw$. By Definition~\ref{def:twin-building}$(\rm ii)$--$(\rm iii)$, $\delta^*(d',c_-) = w$ for all other chambers $d'\in\mathcal P$. If $\ell(sw)=\ell(w)-1$, this maximal value is attained for the codistance $w$ and therefore $d$ is the unique chamber of
		$\mathcal P$ having codistance $w$ from $c_-$. Again, by Definition~\ref{def:twin-building}$(\rm ii)$--$(\rm iii)$, every chamber of $\mathcal P\setminus\{d\}$ has codistance $sw$. 
	\end{proof}
	
	\subsection{Topological completions}\label{sec:top-completions}
	For the rest of this section (and in fact for the rest of the paper), we let $(W,S)$ be a Coxeter system and $(\Lambda, (U_\alpha)_{\alpha \in \Phi}) $ be a twin root datum. We assume that all root groups are finite and that $W$ is infinite, and we let $(X_+, X_-)$ be the associated twin-building. We keep the notations introduced before on $\Lambda^\dagger$, $T$ and $B_\pm$. For $\varepsilon \in \{+,-\}$, the group $\Aut(X_\varepsilon)$ is endowed with the compact-open topology, which makes it a locally compact second countable and totally disconnected group. The kernel of the action of $\Lambda$ on $X_\varepsilon $ is the centralizer $Z_\Lambda(\Lambda^\dagger) $ and is contained in $T$. In particular the root groups map injectively to $\Aut(X_\pm)$. We define $\Gamma = \Lambda / Z_\Lambda(\Lambda^\dagger)$. 
	
	Following \cite[\S1.2]{CR09} or \cite[\S6.1]{CR-RGD}, let
	$G_\varepsilon\leq\Aut(X_\varepsilon)$ be the closure of the image of
	$\Gamma$, and denote the resulting injective dense homomorphism by
	$\pi_\varepsilon\colon\Gamma\to G_\varepsilon$
	\cite[Proposition~1]{CR09}.  The subgroup
	$K_\varepsilon=\Stab_{G_\varepsilon}(c_\varepsilon)$ is compact open and is
	the closure of the image of $B_\varepsilon$.  The completion inherits a
	$BN$-pair with Borel subgroup $K_\varepsilon$, and its associated building is
	$X_\varepsilon$. Let $q_{\min}=\min_{s\in S}q_s$  and $W(t)=\sum_{w\in W}t^{\ell(w)}$.

	\begin{proposition}
		\label{prop:rgd-twin-lattice}
		Assume that $Z_\Lambda(\Lambda^\dagger)$ is finite and that
		$W(1/q_{\min})<\infty$. Then the diagonal embedding identifies $\Gamma$
		with an irreducible lattice in $G_+\times G_-$. 
	\end{proposition}
	\begin{proof}
		By \cite[Proposition~2]{CR09}, the diagonal image is a lattice. Passing to
		the effective quotient by the finite kernel
		$Z_\Lambda(\Lambda^\dagger)$ identifies that image with $\Gamma$. Its two
		coordinate projections are dense by the definition of $G_\pm$, so the
		lattice is irreducible.
	\end{proof}
	We call \emph{twin building lattices} the groups satisfying the assumptions of Proposition~\ref{prop:rgd-twin-lattice}. In fact these lattices are non-uniform. Indeed, $K_+\times K_-$ is compact open and $(G_+ \times G_-)/(K_+ \times K_-)$ is the discrete set of pairs of chambers.  If $\Gamma\backslash (G_+ \times G_-)$ were compact, its image
	$\Gamma\backslash (G_+ \times G_-)/(K_+ \times K_-)$ would be both compact and discrete, hence finite. On the other hand, the Birkhoff decomposition identifies this double quotient
	with the infinitely many $\Gamma$-orbits indexed by $W$, a contradiction.

	\subsection{Split Kac--Moody groups}
	
	Let \(A=(a_{st})_{s,t\in S}\) be an indecomposable generalized Cartan matrix and let \((W,S)\) be the associated crystallographic Coxeter system. We assume that \((W,S)\) is of \emph{compact-hyperbolic type}: \(W\) is an irreducible cocompact reflection group of hyperbolic space, and the standard parabolic subgroup \(W_J<W\) is finite for every proper subset \(J\subsetneq S\). Its rank belongs to \(\{3,4,5\}\). 
	
	Let \(\Lambda=\mathscr G_A(\F_q)\) be the associated minimal simply connected
	split Kac--Moody group. Then $\Lambda$ admits a twin root datum with
	$U_\alpha\cong(\F_q,+)$ for every root $\alpha$. Thus $q_s=q$ and every
	panel has $q+1$ chambers. As before, we denote by \((X_+,X_-)\) its twin
	building and by \(G_\varepsilon\) the closure of the image of $\Lambda$ in
	$\Aut(X_\varepsilon)$.
	These twin building lattices will be our main object of study for the rest of the paper. We collect some results. 
	\begin{proposition}\label{prop:KM-inputs}
		Assume that \(q\geq11\), or that \(W\) is \(3\)-spherical and \(q\geq4\). Then:
		\begin{enumerate}[$(\rm i)$]
			\item \(\Gamma\) is an irreducible lattice with dense injective projections in \(G_+\times G_-\).
			\item \(\Gamma\) is infinite, simple and finitely presented.
			\item \(G_\pm\) are second countable, noncompact, nondiscrete, topologically simple totally disconnected locally compact groups acting strongly transitively and geometrically on \(X_\pm\).
			\item \(G_\pm\) have property~{\rm (T)} and the Howe--Moore property, and \(\Gamma\) has property~{\rm (T)}.
		\end{enumerate}
	\end{proposition}
	
	\begin{proof}
		Since \(q>3\), \cite[Theorem~20]{CR09} and the discussion that follows show
		that $\Lambda=\Lambda^\dagger$ and that $\Gamma=\Lambda/Z(\Lambda)$ is
		simple. In particular, $Z_\Lambda(\Lambda^\dagger)=Z(\Lambda)$ is finite,
		because it is contained in the finite split torus $T$. Together with
		\cite[Lemma~5.1]{CT26}, which gives $W(1/q)<\infty$, Proposition~\ref{prop:rgd-twin-lattice}
		proves $(\rm i)$. The group is finitely presented by \cite{AM97}. See also the
		proof of \cite[Corollary~21]{CR09}. To prove $(\rm iii)$, observe first that the
		root groups are abelian. Split Kac--Moody root data satisfy property
		\textup{(FPRS)} by \cite[Proposition~4]{CR09}. Since
		$\Lambda=\Lambda^\dagger$ and $W$ is infinite and irreducible,
		\cite[Lemma~9]{CR09} shows that the commutator subgroup is dense in each
		completion. Thus \cite[Proposition~11]{CR09}, applied to the effective
		completions, gives their topological simplicity. Strong transitivity follows
		from the completed $BN$-pairs, while local finiteness and chamber
		transitivity make the actions proper and cocompact. The remaining
		topological assertions follow from \cite[Proposition~1]{CR09}. Finally,
		the results \cite[Proposition~4.1 and Corollary~D]{CT26} give $(\rm iv)$ (recall
		that property~{\rm (T)} passes to lattices). 
	\end{proof}
	
	\section{Essential freeness at infinity}\label{sec:essential-freeness}
	
	In this section we prove harmonic essential freeness for the root-generated
	compact-hyperbolic twin-building lattices satisfying the assumptions below.
	The split Kac--Moody groups of Proposition~\ref{prop:KM-inputs} are the main examples.
	
	Let $(\Lambda, (U_\alpha)_{\alpha \in \Phi})$ be a twin root datum with Coxeter system $(W, S)$. We denote by $(X_+, X_-)$ the associated twin building. We keep the notations introduced in Section~\ref{sec:twin}. We denote the two Weyl-distance maps by $\delta_\pm$, the twinning codistance by $\delta^*$ and the length function on $W$ by $\ell$. For the rest of this section, we make the following assumptions: 
	\begin{itemize}
		\item $(W,S)$ is of compact hyperbolic type, and is non-elementary.
		\item The root groups $(U_\alpha)$ are finite, $Z_\Lambda(\Lambda^\dagger)$ is finite and $W(1/q_{\min}) < \infty$.
		\item $\Lambda=\Lambda^\dagger Z_\Lambda(\Lambda^\dagger)$, or equivalently the effective group $\Gamma$ is generated by the images of the rank-one groups $L_s$.
	\end{itemize}
	In particular, by Proposition~\ref{prop:rgd-twin-lattice}, the effective quotient $\Gamma = \Lambda/ Z_\Lambda(\Lambda^\dagger)$ embeds as a non-uniform lattice in $G_+ \times G_-$, where $G_+$ and $G_-$ are the topological completions. For $\varepsilon \in \{+,-\}$, the projection $\pi_\varepsilon \colon \Gamma \to G_\varepsilon $ is injective and dense. Note that these conditions hold for the split Kac--Moody groups of compact hyperbolic type from Proposition~\ref{prop:KM-inputs}.

	We fix opposite fundamental chambers $c_+\in\Ch(X_+)$ and $c_-\in\Ch(X_-)$ and write $B_\varepsilon=\Stab_\Gamma(c_\varepsilon)$ and $K_\varepsilon=\Stab_{G_\varepsilon}(c_\varepsilon)$. The group $K_\varepsilon$ is compact open and $\pi_\varepsilon(B_\varepsilon)\leq K_\varepsilon$. 
	We let $o_\varepsilon$ be a barycenter of $c_\varepsilon$. Since the
	actions are type-preserving, $o_\varepsilon$ is fixed by $K_\varepsilon$.
	
	We use the same notation $L_s$ as before for the image in $\Gamma$ of the rank-one group $\langle U_{\alpha_s},U_{-\alpha_s}\rangle$. These groups are finite by Lemma~\ref{lem:rgd-rank-one-finite}, act transitively on both fundamental $s$-panels, and generate $\Gamma$. 
	
	\subsection{Gromov-hyperbolic geometry}
	
	Each building $X_\varepsilon$ admits a locally compact $\mathrm{CAT}(-1)$ realisation associated with the compact-hyperbolic Coxeter system $(W,S)$. The apartments are isometric to hyperbolic space. Since $W$ is non-elementary, it admits two independent loxodromic isometries, and the actions $\Gamma\curvearrowright X_\varepsilon$ are
	then non-elementary.
	
	We gather some basic properties of these realizations, see for instance \cite[\S3.4]{Re09} and \cite[Chapter~II]{BH99}. 
	
	\begin{lemma}
		\label{lem:ef-realization-geometry}
		Let $\varepsilon \in \{+,-\}$. Then:
		\begin{enumerate}[$(\rm i)$]
			\item $X_\varepsilon$ is proper and every nonconstant geodesic segment extends
			to a complete geodesic line.
			\item $\partial X_\varepsilon$ is compact and metrizable.
			\item An isometry fixing $\partial X_\varepsilon$ pointwise is the identity.
		\end{enumerate}
		The topological completion $G_\varepsilon$ is a closed subgroup of $\Isom(X_\varepsilon)$ with the compact-open topology.
	\end{lemma}
	
	The next lemma shows that an element fixing a large part of the boundary is close to the identity. It is standard but we include a proof for the convenience of the reader. 
	
	\begin{lemma}
		\label{lem:ef-localization}
		Let \(X\) be a proper geodesic Gromov-hyperbolic space with \(|\partial X|\geq3\), and let \(G\leq\Isom(X)\) be closed and act faithfully on \(\partial X\). If \(\nu\) is a (nonatomic) full-support probability measure on \(\partial X\), then
		\[
		\nu\bigl(\Fix_{\partial X}(v_j)\bigr)\longrightarrow1
		\quad\Longrightarrow\quad
		v_j\longrightarrow e
		\]
		for every sequence \((v_j)\) in \(G\).
	\end{lemma}
	
	\begin{proof}
		Put \(M=\partial X\) and let \(M^{(3)}\subset M^3\) be the space of ordered triples of distinct points. We use the standard fact that the diagonal action of \(G\) on \(M^{(3)}\) is proper, see \cite[Corollary~11.133]{drutu_kapovich18}.
		
		The function \(f(v)=\nu(\Fix_M(v))\) is upper semicontinuous. Indeed, if \(v_j\to v\), then \(\limsup_j\mathbf 1_{\Fix_M(v_j)} \leq\mathbf 1_{\Fix_M(v)}\) pointwise. Applying Fatou's lemma to \(1-\mathbf 1_{\Fix_M(v_j)}\) gives
		\[
		\limsup_j\nu(\Fix_M(v_j))
		\leq
		\int_M\limsup_j\mathbf 1_{\Fix_M(v_j)}\,\dd\nu
		\leq\nu(\Fix_M(v)).
		\]
		
		Since \(\nu\) is nonatomic, \(\nu^{\otimes3}(M^{(3)})=1\). Choose a compact \(C\subset M^{(3)}\) with \(\nu^{\otimes3}(C)>1/2\). If \(f(v_j)\to1\), then \(\nu^{\otimes3}(\Fix_M(v_j)^3)=f(v_j)^3>1/2\) for all large \(j\). Hence $C\cap\Fix_M(v_j)^3\neq\varnothing$, and therefore \(v_jC\cap C\neq\varnothing\). Properness on \(M^{(3)}\) shows that the sequence \((v_j)\) is relatively compact.
		
		If \(v\) is a cluster point, upper semicontinuity gives \(f(v)=1\). Thus \(\Fix_M(v)\) is a closed set of full measure. Since \(\nu\) has full support, it equals \(M\). Faithfulness yields \(v=e\). Therefore \(e\) is the only cluster point of the relatively compact sequence \((v_j)\), so \(v_j\to e\).
	\end{proof}
	
	\subsection{A panel-isotropic random walk}
	We are ready to construct random walks on $\Gamma$ and $G_\pm$ with the appropriate properties. For a finite group $L$, let $u_L$ denote its uniform probability measure. Choose numbers $r_s>0$ with $\sum_{s\in S}r_s=1$ and define
	\begin{align*}
		\sigma=\sum_{s\in S}r_su_{L_s}\in\Prob(\Gamma).
	\end{align*}
	The measure $\sigma$ is symmetric, finitely supported and generating. 
	
	As before let $\mathcal P_s^+$ be the $s$-panel containing $c_+$ and put $P_s^+=\Stab_{G_+}(\mathcal P_s^+)$. This is a compact open subgroup of $G_+$. Since it contains the image of $L_s$, it acts transitively on the finite set $\mathcal P_s^+$.  If $m_{P_s^+}$ denotes its normalized Haar measure, define
	\begin{align*}
		\mu_+=\sum_{s\in S}r_sm_{P_s^+}\in\Prob(G_+).
	\end{align*}
	Notice that the same lattice measure $\sigma$ may be used on the two sides: every finite rank-one group $L_s$ is transitive on both fundamental $s$-panels.  More explicitly, if $\mathcal P_s^-$ is the $s$-panel through $c_-$ and $P_s^-=\Stab_{G_-}(\mathcal P_s^-)$, then the negative ambient measure is
	\begin{equation*}
		\mu_-=\sum_{s\in S}r_sm_{P_s^-}.
	\end{equation*}

	Recall that $K_+=\Stab_{G_+}(c_+)$ is the closure of the image of $B_+$, as in Section~\ref{sec:top-completions}. 
	
	\begin{lemma}
		\label{lem:ef-panel-nonsingular}
		Set $F=\bigcup_{s\in S}P_s^+$, and normalize the left Haar measure $m$ on $G_+$ by $m(K_+)=1$. Then $m(P_s^+)=q_s+1$ and 
		\[
		\frac{\dd\mu_+}{\dd m}
		=\sum_{s\in S}\frac{r_s}{q_s+1}\mathbf1_{P_s^+}.
		\]
		The measure $\mu_+$ is symmetric, compactly supported, has bounded density and is $K_+$-bi-invariant. Moreover $G_+=\bigcup_{n\geq1}F^n$, and every $\mu_+$-stationary probability $\nu$ on a Borel $G_+$-space is $K_+$-invariant and $G_+$-quasi-invariant.
	\end{lemma}
	\begin{proof}
		Since \(K_+\leq P_s^+\) and \(P_s^+\) acts transitively on the fundamental \(s\)-panel with chamber stabilizer \(K_+\), we have \([P_s^+:K_+]=q_s+1\). Thus \(m(P_s^+)=q_s+1\) and \(m_{P_s^+}=(q_s+1)^{-1}\mathbf 1_{P_s^+}m\), which gives the
		asserted density. The remaining properties of \(\mu_+\) follow because each \(P_s^+\) is a compact subgroup containing \(K_+\), and \(S\) is finite.
		
		The subgroup \(H=\langle P_s^+\mid s\in S\rangle\) contains the open subgroup \(K_+\), so it is open and closed. It contains every rank-one subgroup \(L_s\), hence also the dense image of \(\Gamma=\langle L_s\mid s\in S\rangle\). Therefore \(H=G_+\). Since \(F\) is symmetric and contains the identity, \(G_+=\bigcup_{n\geq1}F^n\).
		
		Let \(\nu\) be \(\mu_+\)-stationary. For \(k\in K_+\), left \(K_+\)-invariance of \(\mu_+\) gives \(k_*\nu=(k_*\mu_+)*\nu=\mu_+*\nu=\nu\), hence \(\nu\) is \(K_+\)-invariant.
		
		Fix \(g\in G_+\), and choose \(n\) such that \(g\in F^n\). Since \(F\) is compact and \(\supp(\mu_+)=F\), we have \(\supp(\mu_+^{*n})=F^n\). Hence \(\mu_+^{*n}(gK_+)>0\). Right \(K_+\)-invariance implies that, for some \(p>0\), the restriction of \(\mu_+^{*n}\) to \(gK_+\) is \(p\,g_*m_{K_+}\). Consequently,
		\[
		\nu=\mu_+^{*n}*\nu
		\geq p\,(g_*m_{K_+})*\nu
		=p\,g_*\nu,
		\]
		where the last equality uses \(K_+\)-invariance of \(\nu\). Thus \(g_*\nu\ll\nu\). Applying the same argument to \(g^{-1}\) yields the reverse absolute continuity, so \(g_*\nu\sim\nu\).
	\end{proof}
	
	In the following, we always assume that the left Haar measure $m$ on $G_+$ is normalized by $m(K_+)=1$. We chose the measures $\sigma$ and $\mu_+$ so that the induced Markov chains on $\Ch(X_+)$ are the same, as the following lemma shows.  
	\begin{lemma}
		\label{lem:ef-same-chamber-walk}
		Let \((W_n^\sigma)\) and \((W_n^\mu)\) be the right random walks on \(\Gamma\) and \(G_+\), driven respectively by \(\sigma\) and \(\mu_+\). Then the projected processes \((W_n^\sigma c_+)\) and \((W_n^\mu c_+)\) are Markov chains on \(\Ch(X_+)\) with the same transition kernel. If \(\mathcal P_s(d)\) denotes the \(s\)-panel containing \(d\), this kernel is
		\begin{equation}\label{eq:kernel_rw}
			p(d,d')
			=\sum_{s\in S}\frac{r_s}{q_s+1}
			\mathbf 1_{\mathcal P_s(d)}(d').
		\end{equation}
		In particular, the corresponding orbit processes based at \(o_+\) have
		the same law.
	\end{lemma}
	
	\begin{proof}
		Write \(d=gc_+\). Conditional on choosing a panel \(s\), the \(\sigma\)-walk moves from \(d\) to \(g\ell c_+\), where \(\ell\) is uniform in \(L_s\). Since \(L_sc_+=\mathcal P_s^+\) and \(L_s\) acts transitively on this panel, \(\ell c_+\) is uniform on its \(q_s+1\) chambers. Thus the next chamber is uniform on \(g\mathcal P_s^+=\mathcal P_s(d)\).
		
		Now consider the \(\mu_+\)-walk. Again, conditional on choosing the panel \(s\), the increment has normalized Haar distribution on \(P_s^+\). Its image on the finite set \(P_s^+c_+=\mathcal P_s^+\) is likewise uniform. Both walks therefore have the kernel given by Equation~\ref{eq:kernel_rw}.
		
		Finally, the map \(gc_+\mapsto go_+\) is well defined because \(K_+\) fixes \(o_+\). The equality in law of the chamber processes hence implies the equality in law of the corresponding orbit processes.
	\end{proof}
	
	For a chamber $c=gc_+$, set $o_c=go_+$. This is well defined. In the next lemma, recall that we denote by $d_{\rm gal}$ the gallery distance on $\Ch(X_+)$ as in Section~\ref{sec:twin-building} and by $d_X$ the usual CAT$(-1)$ metric on $X_+$. 
	\begin{lemma}
		\label{lem:ef-gallery-growth}
		There exist constants \(A\geq1\) and \(B\geq0\), depending only on the Coxeter tessellation, such that for all chambers \(c,d\),
		\[
		A^{-1}d_{\mathrm{gal}}(c,d)-B
		\leq d_X(o_c,o_d)
		\leq A\,d_{\mathrm{gal}}(c,d)+B.
		\]
		In particular, there exist \(C,a>0\) such that, for every \(R\geq0\),
		\[
		m\{g\in G_+\mid d_X(o_+,go_+)\leq R\}\leq Ce^{aR}.
		\]
		In other words, the action \(G_+\curvearrowright X_+\) has bounded exponential growth.
	\end{lemma}
	
	\begin{proof}
		Any two chambers lie in a common apartment. In the model apartment \(\Sigma\), the Coxeter group \(W\) acts properly and cocompactly, and the gallery distance is the word metric associated with \(S\). The Švarc--Milnor lemma therefore shows that the orbit map
		\(W\to\Sigma\), \(w\mapsto wo\), is a quasi-isometry. Since apartments are isometrically embedded in \(X_+\), the same constants give the stated comparison for every pair of chambers.
		
		The displacement set in the second assertion is a disjoint union of cosets \(gK_+\), each of Haar measure one, indexed by chambers \(gc_+\) whose barycentres lie in the metric ball of radius \(R\) about \(o_+\). By the first assertion, all such chambers lie in a gallery ball of radius at most \(A(R+B)\). The chamber graph has uniformly bounded degree, so its gallery balls have at most exponential growth. Hence their cardinalities, and therefore the Haar measures of the displacement sets, are bounded by \(Ce^{aR}\) for suitable \(C,a>0\).
	\end{proof}
	
	The next proposition provides some properties of the hitting measure associated to the random walk on $X_+$. 
	\begin{proposition}
		\label{prop:ef-common-hitting}
		There exists a nonatomic probability measure \(\nu_+\) of full support on
		\(\partial X_+\) such that:
		\begin{enumerate}[$(\rm i)$]
			\item \((\partial X_+,\nu_+)\) is the Poisson boundary of
			\((G_+,\mu_+)\).
			\item \(\nu_+\) is the hitting measure and the unique
			\(\sigma\)-stationary probability measure on \(\partial X_+\).
		\end{enumerate}
		Moreover, \(\nu_+\) is \(K_+\)-invariant and \(G_+\)-quasi-invariant.
	\end{proposition}
	
	\begin{proof}
		Let \((W_n^\sigma)\) and \((W_n^\mu)\) be the right random walks driven
		by \(\sigma\) and \(\mu_+\), respectively, and set
		\(C_n^\sigma=W_n^\sigma c_+\) and \(C_n^\mu=W_n^\mu c_+\). By
		Lemma~\ref{lem:ef-same-chamber-walk}, the two chamber-valued processes
		\((C_n^\sigma)\) and \((C_n^\mu)\) have the same law, and hence the corresponding \(X_+\)-valued orbit processes also have the same law.
		
		The measure \(\sigma\) has finite support and generates \(\Gamma\), whose action on \(X_+\) is nonelementary. By \cite[Theorems~1.1 and~1.2]{MT18}, the \(\sigma\)-walk converges almost surely to \(\partial X_+\), has positive drift, and its hitting measure is nonatomic and is the unique \(\sigma\)-stationary probability measure on \(\partial X_+\).
		
		As the chamber-processes have the same law, the \(\mu_+\)-random walk also has positive drift. By Lemma~\ref{lem:ef-panel-nonsingular}, \(\mu_+\) is spread out, has compact support and finite first moment, and its support generates \(G_+\). Moreover, \(X_+\) is proper hyperbolic, \(G_+\) is closed in \(\Isom(X_+)\), and Lemma~\ref{lem:ef-gallery-growth} gives bounded exponential growth. Thus \cite[Theorem~1.4]{FT22}, applied to the locally compact group \(G_+\), identifies the Poisson boundary of \((G_+,\mu_+)\) with \(\partial X_+\) equipped with its hitting measure.
		
		Since the two orbit processes have the same law, their boundary limits have the same distribution. Denote this common hitting measure by \(\nu_+\). We have proved $(\rm i)$, $(\rm ii)$ and non-atomicity. 
		
		It remains to prove that $\nu_+$ has full support. Let \(C=\supp(\nu_+)\). Since \(\sigma\) is symmetric and every element of its support has positive mass, \(\sigma\)-stationarity implies \(gC=C\) for every \(g\in\supp(\sigma)\). Hence \(C\) is \(\Gamma\)-invariant and, by density, \(G_+\)-invariant. The action of \(G_+\) on \(X_+\) is cobounded and nonelementary, so its limit set is \(\partial X_+\) and its action there is minimal. Therefore \(C=\partial X_+\).
		
		Finally, \(\nu_+\) is \(\mu_+\)-stationary, so Lemma~\ref{lem:ef-panel-nonsingular} makes it \(K_+\)-invariant and \(G_+\)-quasi-invariant.
	\end{proof}
	
	\subsection{The recurrent codistance chain}
	
	Let $(Y_n)_{n\geq1}$ be independent $\Gamma$-valued random variables of law $\sigma$, and denote by $W_n=Y_1\cdots Y_n$ the $\sigma$-random walk starting at $e$. The chamber sequence $(C_n= W_nc_+)_n$ has the kernel of Lemma~\ref{lem:ef-same-chamber-walk}.  Define its codistance from the fixed negative chamber by
	\begin{equation*}
		Q_n=\delta^*(C_n,c_-)\in W.
	\end{equation*}
	We use the convention of \cite[Definition~5.133]{AB08}, under which an $s$-move in the positive building $X_+$ changes the codistance by left multiplication. With the opposite convention, every occurrence of $sw$ below is replaced by $ws$. For $s\in S$, let $\widehat P_s$ be the following kernel on $W$:
	\begin{align}
		\widehat P_s(w,sw)
		&=
		\begin{cases}
			(q_s+1)^{-1},&\ell(sw)=\ell(w)+1,\\
			q_s(q_s+1)^{-1},&\ell(sw)=\ell(w)-1,
		\end{cases}
		\label{eq:ef-Hecke-offdiag}\\
		\widehat P_s(w,w)&=1-\widehat P_s(w,sw).
		\label{eq:ef-Hecke-diag}
	\end{align}
	
	\begin{lemma}
		\label{lem:ef-codistance-kernel}
		The process $(Q_n)$ is a Markov chain with transition kernel
		\[
		\widehat P=\sum_{s\in S}r_s\widehat P_s.
		\]
	\end{lemma}
	
	\begin{proof}
		We chose the measure $\sigma$ with uniform measures on the rank-one subgroups $L_s$. Equations \eqref{eq:ef-Hecke-offdiag} and \eqref{eq:ef-Hecke-diag} therefore follow from Lemma~\ref{lem:rgd-panel-codistance}. In particular the conditional distribution depends only on $Q_n$, proving the Markov property.
	\end{proof}

	\begin{proposition}
		\label{prop:ef-codistance-recurrence}
		The codistance chain \((Q_n)\) is irreducible and positive recurrent.
		In particular,
		\[
		\mathbf P\bigl[Q_n=e\ \text{infinitely often}\bigr]=1.
		\]
	\end{proposition}
	
	\begin{proof}
		Set \(\rho(w)=q_w^{-1}\). If \(\ell(sw)=\ell(w)+1\), then \(q_{sw}=q_sq_w\), and the definition of \(\widehat P_s\) gives
		\[
		\rho(w)\widehat P_s(w,sw) =\rho(sw)\widehat P_s(sw,w).
		\]
		Thus \(\rho\) is reversible for every \(\widehat P_s\), and hence for \(\widehat P=\sum_s r_s\widehat P_s\). Moreover,
		\[
		\sum_{w\in W}\rho(w)
		\leq\sum_{w\in W}q_{\min}^{-\ell(w)}
		=W(q_{\min}^{-1})<\infty,
		\]
		so \(\rho\) can be normalized to an invariant probability measure on $W$.
		The chain is irreducible. Indeed, if \(w\neq e\), the transition probability from $w$ to $sw$ with \(\ell(sw)=\ell(w)-1\) is positive. By iterating, one reaches \(e\) in finite time. Since every such transition also has positive probability in the reverse direction, the chain starting at \(e\) can reach every \(w \in W\).
		
		An irreducible countable-state Markov chain admitting an invariant probability measure is positive recurrent. Hence every state is visited infinitely often almost surely. Applying this to \(Q_0=e\) proves the claim.
	\end{proof}
	
	\subsection{The conjugation martingale}
	
	Let $\xi_+:\Omega\to\partial X_+$ be the almost-sure limit of $W_no_+$ given by Proposition~\ref{prop:ef-common-hitting}. Its law is $\nu_+$.  For $\gamma\in\Gamma$, set
	\[
	\chi_n(\gamma)=\pi_+(W_n^{-1}\gamma W_n)\in G_+.
	\]
	We also define
	\[
	f_+(v)=\nu_+\bigl(\Fix_{\partial X_+}(v)\bigr),
	\]
	for all $v\in G_+$. 
	Before proving that the $\Gamma$-action on $(\partial X_+, \nu_+)$ is essentially free, we shall need the following result on the conjugated walk $\chi_n(\gamma)$. 
	\begin{lemma}
		\label{lem:ef-conjugation-martingale}
		For every $\gamma\in\Gamma$,
		\[
		\mathbf E\!\left[
		\mathbf 1_{\{\gamma\xi_+=\xi_+\}}
		\,\middle|\,Y_1,\ldots,Y_n
		\right]
		=f_+\bigl(\chi_n(\gamma)\bigr).
		\]
		Consequently, almost surely on $\{\gamma\xi_+=\xi_+\}$ one has
		\begin{equation}
			\chi_n(\gamma)\longrightarrow e.
			\label{eq:ef-conjugates-contract}
		\end{equation}
	\end{lemma}
	
	\begin{proof}
		Let $\xi_+^{(n)}$ be the limit of the tail walk $Y_{n+1}\cdots Y_mo_+$.  It is independent of $\sigma(Y_1,\ldots,Y_n)$, has law $\nu_+$, and $\xi_+=W_n\xi_+^{(n)}$.  Hence
		\[
		\{\gamma\xi_+=\xi_+\}
		=\{W_n^{-1}\gamma W_n\xi_+^{(n)}=\xi_+^{(n)}\},
		\]
		which gives the conditional-expectation formula. By Doob's martingale convergence theorem, the left-hand side converges almost surely to $\mathbf 1_{\{\gamma\xi_+=\xi_+\}}$.  On the event $\{\gamma\xi_+=\xi_+\}$, we have $f_+(\chi_n(\gamma))\to1$, so Lemma~\ref{lem:ef-localization} gives \eqref{eq:ef-conjugates-contract}.
	\end{proof}
	We are now ready for the main result of this section. 
	\begin{theorem}
		\label{thm:ef-panel-essential-free}
		For every $\gamma\in\Gamma\setminus\{e\}$,
		\begin{equation}
			\nu_+\bigl(\Fix_{\partial X_+}(\gamma)\bigr)=0.
			\label{eq:ef-plus-fixed-null}
		\end{equation}
		Consequently the action $\Gamma\curvearrowright(\partial X_+,\nu_+)$ is essentially free.
	\end{theorem}
	
	\begin{proof}
		Fix $\gamma\neq e$ and suppose that the event $A_\gamma=\{\gamma\xi_+=\xi_+\}$ has positive probability. Proposition~\ref{prop:ef-codistance-recurrence}
		holds on a conull event, while \eqref{eq:ef-conjugates-contract} holds almost surely on $A_\gamma$. We can then choose a sample path in the intersection of $A_\gamma$ with these two conull events relative to $A_\gamma$. In particular there exist times $n_j\to\infty$ such that $Q_{n_j}=e$. As the codistance is $\Gamma$-invariant,
		\[
		\delta^*(c_+,W_{n_j}^{-1}c_-) =\delta^*(W_{n_j}c_+,c_-)=e.
		\]
		Thus $W_{n_j}^{-1}c_-$ is opposite $c_+$. As $B_+$ acts transitively on the set of negative chambers opposite $c_+$, we may choose $b_j\in B_+$ such that $b_jW_{n_j}^{-1}c_-=c_-$ for all $j$. Set $h_j=b_jW_{n_j}^{-1}\in B_-$ and $a_j=h_j\gamma h_j^{-1}\in\Gamma$. In $G_+$, we have that
		\[
		\pi_+(a_j)
		=\pi_+(b_j)\chi_{n_j}(\gamma)\pi_+(b_j)^{-1}.
		\]
		The elements $\pi_+(b_j)$ belong to the compact group $K_+$. As conjugation on $K_+$ is uniformly continuous, Lemma~\ref{lem:ef-conjugation-martingale} implies that
		\begin{equation}
			\pi_+(a_j)\longrightarrow e.
			\label{eq:ef-plus-aj}
		\end{equation}
		On the other hand, $h_j\in B_-$ gives that
		\begin{equation}
			\pi_-(a_j)
			\in K_-\pi_-(\gamma)K_-,
			\label{eq:ef-minus-aj}
		\end{equation}
		and the latter set is compact.
		
		The set consisting of the convergent sequence $(\pi_+(a_j))_j$ and its limit is compact. Together with Equation~\eqref{eq:ef-minus-aj}, this places every $(\pi_+(a_j),\pi_-(a_j))$ in a fixed compact subset of $G_+\times G_-$.  The diagonal image of $\Gamma$ is discrete, so this compact set meets $\Gamma$ in a finite set.  Passing to a subsequence, there exists $a\in\Gamma$ such that $a_j = a$ for all $j$ large enough. Equation~\eqref{eq:ef-plus-aj} gives $\pi_+(a)=e$, and injectivity of
		$\pi_+|_\Gamma$ gives $a=e$. But $a_j$ is conjugate to the nontrivial element $\gamma$, a contradiction.  This proves Equation~\eqref{eq:ef-plus-fixed-null}.
		
		Finally, as $\Gamma$ is countable,
		\[
		\nu_+\bigl(\{\xi\mid\Gamma_\xi\neq\{e\}\}\bigr)
		\leq\sum_{\gamma\neq e}
		\nu_+(\Fix_{\partial X_+}(\gamma))=0
		\]
		and the action is essentially free. 
	\end{proof}
	
	\subsection{A concrete \texorpdfstring{$\rC^*$}{C*}-simple measure}
	
	We can use the concrete properties of $\sigma$ and the harmonic measure $\nu_+$ to say more about the $\rC^*$-algebraic properties of $\Gamma$. These results are not needed for Theorem~\ref{thm:KM-main-intro}, but can be of independent interest. Following \cite[Definition~5.3]{HK23}, a probability measure $\eta$ on a countable group is \emph{$\rC^*$-simple} if the canonical trace is the unique $\eta$-stationary state on the reduced group $\rC^*$-algebra for the conjugation action.
	
	\begin{corollary}\label{cor:ef-sigma-cstar}
		The measure $\sigma$ is $\rC^*$-simple. 
	\end{corollary}
	
	\begin{proof}
		By \cite[Proposition~4.16]{MT18}, $W_n\nu_+$ converges almost surely to $\delta_{\xi_+}$. Together with Proposition~\ref{prop:ef-common-hitting}, this says that $(\partial X_+,\nu_+)$ is a uniquely $\sigma$-stationary $\sigma$-boundary (a $\sigma$-\textup{USB} in the terminology of \cite{HK23}). Its action is essentially free by Theorem~\ref{thm:ef-panel-essential-free}. Hence \cite[Theorem~4.12]{HK23} shows that $\sigma$ is $\rC^*$-simple. 
	\end{proof}
	As the split Kac--Moody groups from Theorem~\ref{thm:KM-main-intro} satisfy the assumptions of this section, note that Corollary~\ref{cor:cstar-measure-intro} is a special case of Corollary~\ref{cor:ef-sigma-cstar}. 
	\begin{remark}
		Since $\sigma$ is generating, the probability measure
		\[
		\widehat\sigma=\sum_{k\geq1}2^{-k}\sigma^{*k}
		\]
		has full support. Moreover, $\widehat\sigma$ and $\sigma$ have the same stationary states on $\rC^*_\lambda(\Gamma)$, so $\widehat\sigma$ is also $\rC^*$-simple. 
	\end{remark}
	
	\begin{corollary}\label{cor:KM-selfless}
		The reduced group $\rC^*$-algebra $\rC^*_\lambda(\Gamma)$ is completely selfless. 
	\end{corollary}
	
	\begin{proof}
		The action of $\Gamma$ on $X_+$ is of general type and has full limit set $\partial X_+$. Hence the action on $\partial X_+$ is minimal and extremely proximal. Thus $\partial X_+$ is an extreme $\Gamma$-boundary. Let
		\[
		Y=\{\xi\in\partial X_+\mid\Gamma_\xi=\{e\}\}.
		\]
		By Theorem~\ref{thm:ef-panel-essential-free}, $\nu_+(Y)=1$. Since $\nu_+$ has full support, every nonempty open subset of $\partial X_+$ meets $Y$. Hence the boundary action is topologically free. The conclusion follows from \cite[Theorems~1 and~11]{Oza26}.
	\end{proof}
	
	\begin{remark}
		This recovers, in the particular case of the compact-hyperbolic twin-building lattices considered here, the selflessness conclusion of \cite[Proposition~H]{Ryb26}, which was obtained via completely different methods, using curtain hyperbolic models and an asymptotic-dimension argument. 
	\end{remark}
	
	\subsection{Charmenability and charfiniteness}
	
	Applying the preceding construction with the two signs interchanged gives admissible measures $\mu_\pm$ on $G_\pm$, with Poisson boundaries $(\partial X_\pm,\nu_\pm)$. Theorem~\ref{thm:ef-panel-essential-free} gives
	\[
	\nu_\varepsilon
	\bigl(\Fix_{\partial X_\varepsilon}(\pi_\varepsilon(\gamma))\bigr)=0
	\]
	for all $\gamma\in\Gamma\setminus\{e\}$ and $\varepsilon\in\{+,-\}$. 
	
	\begin{corollary}
		\label{cor:ef-panel-charmenability}
		Assume moreover that $\Rad(\Gamma)=\{e\}$ and that $G_+$ and $G_-$
		have the Howe--Moore property.
		Then $\Gamma$ is charmenable.
		If moreover $\Gamma$ has property~{\rm(T)}, it is charfinite.
	\end{corollary}
	
	\begin{proof}
		Proposition~\ref{prop:rgd-twin-lattice} gives the lattice hypotheses.
		Proposition~\ref{prop:ef-common-hitting} and
		Theorem~\ref{thm:ef-panel-essential-free}, applied in both signs, give
		admissible random walks with essentially free harmonic boundary actions.
		Lemma~\ref{lem:geometric-CE} supplies conjugacy escape, so
		Theorem~\ref{thm:twin-main} applies.
	\end{proof}
	
	\begin{proof}[Proof of Theorem~\ref{thm:KM-main-intro}]
		We just need to verify that $\Gamma $ satisfies the assumptions of Corollary~\ref{cor:ef-panel-charmenability}. By Proposition~\ref{prop:KM-inputs}, $\Gamma $ is a finitely presented irreducible lattice in $G_+ \times G_-$, is simple, and has property~{\rm (T)}. Moreover $G_\pm$ have the Howe--Moore property. Since $q>3$, the discussion following \cite[Theorem~20]{CR09} gives $\Lambda=\Lambda^\dagger$. Hence Lemma~\ref{lem:rgd-rank-one-finite} shows that the images of the groups $L_s$ generate $\Gamma$. Each $L_s$ acts transitively on both fundamental $s$-panels. Finally, \cite[Lemma~5.1]{CT26} gives $W(1/q)<\infty$. Thus $\Gamma$ satisfies all of the assumptions of this section. Since $\Gamma$ is simple and nonamenable, $\Rad(\Gamma)=\{e\}$. Corollary~\ref{cor:ef-panel-charmenability} shows that $\Gamma$ is
		charmenable and charfinite.
	\end{proof}
	
	\subsection{Concrete rigidity consequences}
	In this section we gather some interesting rigidity properties of the twin-building lattice $\Gamma$. 
	\begin{proof}[Proof of Corollary~\ref{cor:KM-main}]
		By Theorem~\ref{thm:KM-main-intro}, the group $\Gamma$ is charfinite and hence
		charmenable. It is nonamenable because it is infinite and has
		property~{\rm (T)}. Its simplicity therefore implies
		\(\Rad(\Gamma)=\{e\}\). Let
		\(\tau\in\Char(\Gamma)\). By charmenability, either
		\(\tau=\delta_e\), or \(\tau\) is amenable and hence finite dimensional by
		charfiniteness. Every finite dimensional unitary representation of
		\(\Gamma\) is trivial. Indeed, a nontrivial one would be faithful by
		simplicity. Since \(\Gamma\) is finitely generated, Mal'cev's theorem
		\cite{Ma40} would make it residually finite, which is impossible for an
		infinite simple group. Hence, in the latter case, \(\tau=1_\Gamma\). Since an infinite simple group
		is icc, \(\delta_e\) is a character.
		
		We now derive the \textup{IRS} assertion from this character rigidity. Let
		\(\eta\) be an ergodic \textup{IRS} of \(\Gamma\). By
		\cite[Proposition~1.4]{AGV14}, it is the stabilizer distribution of a
		probability-measure-preserving action
		\(\Gamma\curvearrowright(Y,m)\). Since \(\eta\) is ergodic, almost every
		ergodic component has the same stabilizer distribution. We may therefore
		assume that the action is ergodic. If \(Y\) has an atom, then, modulo null
		sets, \(Y\) is a finite transitive \(\Gamma\)-space and \(\eta\)-almost every
		subgroup has finite index.
		
		Suppose that \(Y\) is diffuse. Then almost every orbit is infinite.
		Proposition~3.1 of \cite{PT16} shows that the orbit-relation representation on
		\(\rL^2(\mathcal R)\) is weakly mixing. For every \(g\in\Gamma\), the
		coefficient of its diagonal unit vector is the fixed-point trace
		\(\chi_\eta(g)=m(\{y\in Y\mid gy=y\})
		=\eta(\{H\in\operatorname{Sub}(\Gamma)\mid g\in H\})\).
		Consequently, the GNS representation of \(\chi_\eta\), being a cyclic
		subrepresentation of the orbit-relation representation, is weakly mixing.
		Since \(1_\Gamma\) and \(\delta_e\) are the only extreme traces, the
		Krein--Milman theorem gives
		\(\chi_\eta=t\,1_\Gamma+(1-t)\delta_e\) for some \(t\in[0,1]\).
		If \(t>0\), its GNS representation contains the trivial representation.
		Therefore \(t=0\), so the action is essentially free and
		\(\eta=\delta_{\{e\}}\). This is the argument of
		\cite[Theorem~3.2 and the discussion following Corollary~3.4]{PT16}.
		
		It remains to treat the atomic case. An infinite simple group has no proper
		finite-index subgroup. The core of such a subgroup would be a finite-index
		normal subgroup, and a trivial core would embed \(\Gamma\) into a finite
		symmetric group. Hence \(\eta=\delta_\Gamma\). By ergodic decomposition,
		every \textup{IRS} is a convex combination of these two measures.
		
		Finally, let \(X\subset\operatorname{Sub}(\Gamma)\) be a \textup{URS} of the
		split Kac--Moody group. Since $\Gamma$ is charfinite and has trivial
		amenable radical, \cite[Proposition~3.5]{BBHP22} shows that \(X\) is finite.
		The conjugation
		action on \(X\) defines a homomorphism from \(\Gamma\) to
		\(\operatorname{Sym}(X)\). Simplicity and infinitude force this homomorphism
		to be trivial. Hence every \(H\in X\) is normal in \(\Gamma\). By minimality,
		\(X\) is a singleton, and simplicity gives \(X=\{\{e\}\}\) or
		\(X=\{\Gamma\}\).
	\end{proof}
	
	\begin{proof}[Proof of Corollary~\ref{cor:unitary-intro}]
		As shown in the proof of Corollary~\ref{cor:KM-main}, the group \(\Gamma\)
		is charmenable, has trivial amenable radical, and every finite-dimensional
		unitary representation of \(\Gamma\) is trivial. Moreover, \(\Gamma\) has
		property~{\rm (T)} by Proposition~\ref{prop:KM-inputs}.
		
		Let \(\rho:\Gamma\to\cU(\mathscr H)\) be a nontrivial unitary
		representation. Then \(\mathscr H_0=\mathscr H\ominus\mathscr H^\Gamma\)
		is nonzero, and the restriction \(\rho_0\) of \(\rho\) to \(\mathscr H_0\)
		has no nonzero finite-dimensional subrepresentation. Thus $\rho_0$ is weakly
		mixing. By \cite[Theorem~1]{BV93}, \(\rho_0\) is nonamenable.
		It follows from \cite[Proposition~3.7]{BBHP22} that \(\rho_0\) weakly
		contains \(\lambda_\Gamma\), and hence so does \(\rho\).
		By the definition of weak containment, the map
		\(\rho_0(\gamma)\mapsto\lambda_\Gamma(\gamma)\) extends to a
		surjective $*$-homomorphism
		\[
		\Theta_{\rho_0,\lambda}:\rC^*_{\rho_0}(\Gamma)
		\longrightarrow\rC^*_\lambda(\Gamma).
		\]
		By \cite[Proposition~3.7]{BBHP22},
		\(\tau_\Gamma\circ\Theta_{\rho_0,\lambda}\) is the unique tracial state on
		\(\rC^*_{\rho_0}(\Gamma)\), and every proper ideal is contained in its
		left kernel. Since $\tau_\Gamma$ is faithful on
		$\rC^*_\lambda(\Gamma)$, this left kernel is
		$\ker(\Theta_{\rho_0,\lambda})$. Hence
		$\ker(\Theta_{\rho_0,\lambda})$ is the unique proper maximal ideal.
	\end{proof}

\end{document}